\documentclass[10pt,reqno]{amsart}
\usepackage{amsfonts}
\usepackage{mathrsfs}  
\usepackage{amsthm}
\usepackage{amsxtra}
\usepackage{amssymb}
\usepackage{tikz}
\usepackage{enumitem}
\usepackage{url}
\usepackage{caption}
\usepackage{graphicx}
\usepackage{subfigure}
\usepackage{tabularx}
\usepackage{multirow,booktabs}
\usepackage{float} 
\usepackage[margin=1.55in]{geometry}
\newtheorem{theorem}{Theorem}[section]
\newtheorem{corollary}[theorem]{Corollary}

\newtheorem{lemma}[theorem]{Lemma}

\newtheorem{Assumption}{Assumption}
\newtheorem{remark}[theorem]{Remark}
\newtheorem{example}[theorem]{Example}
\newtheorem{algorithm}{Algorithm}

\newcommand{\R}{\mathbb{R}}

\graphicspath{{image/}}

\def\EE{\mathbb E}
\def\P{\mathbb P}
\def\d{\delta}
\def\l{\langle}
\def\r{\rangle}
\def\F{{\mathcal F}}

\begin{document}
	
\title[ ]
{Convergence rates of Landweber-type methods for inverse problems in Banach spaces}
	

\author{Qinian Jin}
\address{Mathematical Sciences Institute, Australian National
University, Canberra, ACT 2601, Australia}
\email{qinian.jin@anu.edu.au} \curraddr{}
 
	
	
	
\begin{abstract}
Landweber-type methods are prominent for solving ill-posed inverse problems in Banach 
spaces and their convergence has been well-understood. However, how to derive their convergence rates 
remains a challenging open question. In this paper, we tackle the challenge of deriving convergence 
rates for Landweber-type methods applied to ill-posed inverse problems, where forward 
operators map from a Banach space to a Hilbert space. Under a benchmark source condition, we 
introduce a novel strategy to derive convergence rates when the method is terminated by either an
{\it a priori} stopping rule or the discrepancy principle. Our results offer substantial flexibility 
regarding step sizes, by allowing the use of variable step sizes. By extending the strategy to deal with 
the stochastic mirror descent method for solving nonlinear ill-posed systems with exact data, under a
benchmark source condition we also obtain an almost sure convergence rate in terms of the number of 
iterations. 
\end{abstract}
	
\def\d{\delta}
\def\P{\mathbb{P}}
\def\l{\langle}
\def\r{\rangle}
\def\la{\lambda}
\def\EE{{\mathbb E}}
\def\R{{\mathcal R}}
\def\a{\alpha}
\def\p{\partial}
\def\ep{\varepsilon}
\def\PP{{\mathbb P}}

%
		
\maketitle

\section{\bf Introduction}
\setcounter{equation}{0}

Due to their simplicity of implementation and low computational complexity per iteration, Landweber-type 
methods are well-recognized for solving ill-posed inverse problems in Banach spaces. Although 
their convergence has been well-understood, deriving their convergence rates remains a challenging 
open question. In this paper, we will consider Landweber-type methods for solving ill-posed inverse 
problems, where the forward operators map from a Banach space to a Hilbert space, and develop novel 
strategies to tackle the challenge of deriving convergence rates under a benchmark source condition 
on sought solutions. 

We consider ill-posed inverse problems governed by the operator equation
\begin{align}\label{Land.eq}
F(x) = y,
\end{align}
where $F : \mbox{dom}(F)\subset X \to Y$ is a Fr\'{e}chet differentiable operator mapping from a Banach 
spaces $X$ to a Hilbert space $Y$, with domain $\mbox{dom}(F)$. The Fr\'{e}chet derivative of $F$ at 
$x\in \mbox{dom}(F)$ is denoted by $F'(x)$. For numerous examples of inverse problems that take the form 
of equation (\ref{Land.eq}), such as integral equations of the first kind, tomography problems, and 
parameter estimation in partial differential equations, see \cite{EHN1996,Gr1984,N2001,SKHK2012}.

We assume (\ref{Land.eq}) has a solution, i.e. 
$y \in \mbox{Ran}(F)$, the range of $F$. In practical applications, prior feature information about 
the sought solution is often available, and it is important to incorporate this information into the 
reconstruction process. Let $\R: X \to (-\infty, \infty]$ be a proper, lower semi-continuous, 
strong convex function that takes into account the available feature information. 
To reconstruct a solution of (\ref{Land.eq}) that aligns with the features described by $\R$, we start 
with an initial guess $x_0 \in X$ with $\p \R(x_0) \ne \emptyset$ and take $\xi_0 \in \p \R(x_0)$, where 
$\p \R(x_0)$ denotes the subdifferential of $\R$ at $x_0$. We want to determine a solution $x^\dag$ of 
(\ref{Land.eq}) such that 
\begin{align}\label{Land.min}
D_\R^{\xi_0}(x^\dag, x_0) = \min\left\{D_\R^{\xi_0}(x, x_0): F(x) = y\right\},
\end{align}
where $D_\R^{\xi_0}(x, x_0)$ denotes the Bregman distance induced by $\R$ at $x_0$ in the direction 
$\xi_0$; for the definition of subdifferential, Bregman distance and other relevant concepts from 
convex analysis, see \cite{Z2002}. A brief review of these concepts is provided in Section \ref{sect2}.
We are mainly interested in the situation where the problem (\ref{Land.min}) is ill-posed in the sense 
that its solution does not depend continuously on the data. 

In applications, data acquired through experiments inevitably contains noise. Assume, instead of the exact 
data $y$, we only have a noisy data $y^\d$ satisfying 
\begin{align*}
\|y^\d - y\| \le \d
\end{align*}
with a small noise level $\d>0$. Directly substituting $y$ with $y^\d$ in the minimization problem 
(\ref{Land.min}) may render the problem ill-defined, even if well-defined, the solution obtained from 
$y^\d$ may not depend continuously on the data. To effectively utilize noisy data $y^\d$ to approximate 
the solution of (\ref{Land.min}), regularization techniques must be employed to mitigate this instability. 

Although many regularization methods have been developed for solving (\ref{Land.min}) with noisy data, in this paper we will focus on the following Landweber-type method
\begin{align}\label{Land}
\begin{split}
\xi_{k+1}^\d & = \xi_k^\d - \gamma_k^\d F'(x_k^\d)^* (F(x_k^\d) - y^\d),\\
x_{k+1}^\d & = \arg\min_{x\in X} \left\{\R(x) - \l \xi_{k+1}^\d, x\r\right\},
\end{split}
\end{align}
where $x_0^\d := x_0$ and $\xi_0^\d := \xi_0$. This method builds on various efforts to extend the classical 
Landweber iteration (\cite{BH2012,Jin2016,JW2013,N2018,NS1995,RJ2020,S1996,SLS2006}) and can be interpreted 
as a mirror descent method (\cite{JLZ2023,NY1983}). Each iteration involves two steps: the first performs 
a gradient step in the dual space $X^*$ of $X$, followed by the second step, which uses the mirror map 
defined by $\R$ to pull the element back to the original primal space $X$. Under the tangential cone condition 
$$
\|F(\bar x) - F(x) - F'(x)(\bar x - x)\| \le \eta \|F(\bar x) - F(x)\|
$$
around a sought solution with $\eta \in [0, 1)$, the convergence has been proved when the iteration is terminated by {\it a priori} or {\it a posteriori} stopping rules, see \cite{HNS1995,Jin2016,JW2013,S1996} for instance.

To assess how fast the iteration converges towards the desired solution, understanding the convergence 
rates is crucial and has significant theoretical interest. For ill-posed inverse problems, deriving such 
rates typically relies on assuming appropriate source conditions are satisfied by the sought solution.
In Hilbert spaces, the linear Landweber iteration is well-studied, where it is known to be an 
order-optimal regularization method (\cite{EHN1996}). For the more general method (\ref{Land}) 
in Hilbert spaces with a nonlinear forward operator $F$ and a regularization functional 
$\R(x) = \frac{1}{2} \|x\|^2$, convergence rate analyses have been undertaken in \cite{HNS1995,N2018,NS1995}.
Specifically, in the seminal work by \cite{HNS1995}, the nonlinear Landweber iteration 
\begin{align}\label{cLand}
x_{k+1}^\d = x_k^\d - \gamma F'(x_k^\d)^*(F(x_k^\d) - y^\d)
\end{align}
in Hilbert spaces with a constant step size has been considered. Under the range invariance condition 
\begin{align}\label{RI}
F'(x) = Q_x F'(x^\dag) \quad \mbox{ and } \quad \|I - Q_x\| \le \kappa_0 \|x-x^\dag\|
\end{align}
in a neighborhood of the sought solution $x^\dag$, it has been shown that, if the method is terminated 
by the discrepancy principle 
\begin{align}\label{DP0}
\|F(x_{k_\d}^\d) - y^\d\| \le \tau \d < \|F(x_k^\d) - y^\d\|, \quad 0 \le k < k_\d
\end{align}
with suitably large $\tau>1$, then the convergence rate 
$$
\|x_{k_\d}^\d - x^\dag\| = O(\d^{2\mu/(1+2\mu)}) 
$$
holds, provided the source condition 
$$
x^\dag - x_0 = (F'(x^\dag)^* F'(x^\dag))^\mu \omega
$$
is satisfied for some $0< \mu \le 1/2$ with sufficiently small $\|\omega\|$. 

Deriving convergence rates for the method (\ref{Land}) in its full generality is a very challenging 
task. The main difficulties arise from several aspects: the possible nonlinearity of the forward operator
$F$, the non-quadratic nature of the regularization functional $\R$, the non-Hilbertian structure 
of the underlying space $X$, and the possible use of variable step-sizes $\gamma_k^\d$. When $F$
is a bounded linear operator and the step-size is constant, the convergence rate analysis on the 
corresponding method has been rigorously studied in \cite{Jin2022}. In that work, through the 
interpretation of the method as a dual gradient method, the convergence rates have been derived 
successfully under varational source conditions (\cite{HKPS2007}) when the iteration is terminated by either 
an {\it a priori} stopping rule or the discrepancy principle. The approach developed in \cite{Jin2022} 
heavily relies on the linearity of the forward operator and the constancy of the step-size. However, it 
is unclear how to extend the strategy presented in \cite{Jin2022} to address the equation (\ref{Land.eq}) 
when the forward operator is nonlinear or when variable step sizes are employed.

In this paper, we will develop entirely novel ideas to tackle the challenge of deriving convergence rate 
of the method (\ref{Land}) under the benchmark source condition
\begin{align}\label{sc0}
\xi^\dag := \xi_0 + F'(x^\dag)^* \la^\dag \in \p \R(x^\dag) \mbox{ for some } \la^\dag \in Y
\end{align}
on the sought solution $x^\dag$, where $\p \R$ denotes the subdifferential of $\R$. Note that, under 
the tangential cone condition, (\ref{Land.min}) can be equivalently stated as 
\begin{align*}
\min\left\{D_\R^{\xi_0}(x, x_0): F'(x^\dag) (x-x^\dag) = 0\right\}.
\end{align*}
The Karush-Kuhn-Tucker (KKT) condition for this problem is exactly (\ref{sc0}). In the context of ill-posed 
problems in infinite-dimensional spaces, the mere existence of a solution does not guarantee the 
satisfaction of the KKT condition. Therefore, (\ref{sc0}) is commonly referred to as a source condition. 
This source condition has been extensively utilized in deriving convergence rates for regularization 
methods of ill-posed inverse problems (\cite{BO2004,HKPS2007,Jin2022}). Under the range invariance condition 
(\ref{RI}) on $F$ and the source condition (\ref{sc0}) on $x^\dag$, we will show that 
the convergence rate
$$
\|x_{k_\d}^\d - x^\dag\| = O(\d^{1/2}) 
$$
holds if $k_\d$ is chosen by a suitable {\it a priori} stopping rule or the discrepancy principle 
(\ref{DP0}). Our results hold for any strongly convex regularization functional $\R$ with great 
flexibility on step-sizes by allowing the use of variable ones. Thus constant and adaptive 
choices of step-sizes are covered by our results. The key point for deriving our convergence rate
is the observation that, under the condition (\ref{RI}), $\xi_k^\d - \xi_0 \in \mbox{Ran}(F'(x^\dag)^*)$
for every $k$. This motivates us to reformulate the method (\ref{Land}) in an equivalent manner which 
defines an auxiliary sequence $\{\la_k^\d\}$ such that $\xi_k^\d = \xi_0 + F'(x^\dag)^* \la_k^\d$ for 
all $k$. Based on this sequence $\{\la_k^\d\}$, the source condition (\ref{sc0}), and the sequence 
$\{x_k^\d\}$, we can construct a sequence of quantities which obey a nice recursive relation. Through 
the investigation of these quantities leads us to obtain the desired convergence rate under an 
{\it a priori} stopping rule. Based on this {\it a priori} result, after a careful comparison we obtain 
the convergence rate when the method is terminated by the discrepancy principle. Furthermore, we extend 
the idea to deal with the stochastic mirror descent method, which is a stochastic version of the 
method (\ref{Land}) for solving nonlinear ill-posed systems; we consider the exact data case and obtain 
an almost sure convergence rate result under a benchmark source condition. 

This paper is organized as follows. In Section \ref{sect2} we collect some basic facts of convex analysis 
in Banach spaces. In Section \ref{sect3} we devote to deriving convergence rates of the Landweber-type 
method (\ref{Land}) when terminated by either {\it a priori} stopping rules or the discrepancy principle. In 
Section \ref{sect4} we extend the idea to derive convergence rate for the stochastic mirror descent method 
for nonlinear ill-posed systems with exact data. Finally, in Section \ref{sect5} we provide some numerical 
results to validate the theoretical results on convergence rates. 

\section{\bf Preliminaries}\label{sect2}
\setcounter{equation}{0}

In this section, we collect some basic facts on convex analysis in Banach spaces; for more 
details one may refer to \cite{Z2002}.

Throughout the paper, the same notation $\l\cdot,\cdot\r$ will be used to denote either the inner product in 
a Hilbert space or the duality pairing in a Banach space. Let $X$ be a Banach space with norm $\|\cdot\|$, 
we use $X^*$ to denote its dual space. For a convex function $f : X \to  (-\infty, \infty]$,  the set
$$
\mbox{dom}(f) := \{x \in X : f(x) < \infty\}
$$
is called the effective domain of $f$. If $\mbox{dom}(f) \ne \emptyset$, $f$ is called proper. The 
subdifferential of $f$ is the set-valued mapping $\p f: X \rightrightarrows X^*$ defined by 
$$
\p f(x) := \{\xi \in X^*: f(\bar x) - f(x) -\l \xi, \bar x-x\r\ge 0 \mbox{ for all } \bar x\in X\}
$$
for each $x \in X$. Note that $\p f(x)$ could be empty for some $x$. The domain and graph of $\p f$ are 
defined respectively as
\begin{align*}
\mbox{dom}(\p f) := \{x\in \mbox{dom}(f): \p f(x)\ne \emptyset \}
\end{align*}
and 
\begin{align*}   
\mbox{graph}(\p f) 
:= \left\{(x, \xi)\in X\times X^*: x\in \mbox{dom}(\p f) \mbox{ and } \xi \in \p f(x)\right\}.
\end{align*}
Given $x\in \mbox{dom}(f)$, an element $\xi\in \p f(x)$ is called a subgradient of $f$ at $x$. 
The Bregman distance induced by $f$ at $x$ in the direction $\xi \in \p f(x)$ is defined by
$$
D_f^\xi(\bar x, x) := f(\bar x) -  f(x) - \l \xi, \bar x - x\r,  \quad \forall \bar x \in X
$$
which is always nonnegative and satisfies the identity
\begin{equation}\label{Land.20}
D_f^{\xi_2} (x,x_2) - D_f^{\xi_1} (x, x_1) = D_f^{\xi_2} (x_1,x_2) + \l \xi_2-\xi_1, x_1-x\r
\end{equation}
for all $x\in \mbox{dom}(f)$ and $(x_1, \xi_1), (x_2, \xi_2) \in \mbox{graph}(\p f)$.

For a proper function $f : X\to (-\infty, \infty]$, its convex conjugate is defined by
$$
f^*(\xi) :=  \sup_{x\in X}  \{\l \xi, x\r - f(x)\}, \quad  \xi \in X^* 
$$
which is a convex function taking values in $(-\infty, \infty]$. If $f : X\to  (-\infty, \infty]$ 
is proper, lower semi-continuous and convex, $f^*$ is also proper and 
\begin{align}\label{Land.25}
\xi \in \p f(x) \Longleftrightarrow x\in \p f^*(\xi) \Longleftrightarrow  f(x) + f^*(\xi) = \l \xi, x\r.
\end{align}

A proper function $f : X \to (-\infty, \infty]$ is called strongly convex if there exists a
constant $\sigma>0$ such that
\begin{align}\label{Land.23}
f(t\bar x + (1-t) x) + \sigma t(1-t) \|\bar x -x\|^2 \le  tf(\bar x) + (1-t)f(x)
\end{align}
for all $\bar x, x\in \mbox{dom}(f)$ and $t\in [0, 1]$. It is easy to see that if $f : X \to (-\infty, \infty]$ 
is strongly convex in the sense of (\ref{Land.23}), then 
\begin{align}\label{Land.24}
D_f^\xi(\bar x, x) \ge \sigma \|x-\bar x\|^2
\end{align}
for all $\bar x\in \mbox{dom}(f)$, $x\in \mbox{dom}(\p f)$ and $\xi\in \p f(x)$. Furthermore, for 
a proper, lower semi-continuous, strongly convex function $f: X \to (-\infty, \infty]$ satisfying 
(\ref{Land.23}), it is known from \cite[Corollary 3.5.11]{Z2002} that $\mbox{dom}(f^*) = X^*$, 
$f^*$ is Fr\'{e}chet differentiable and its gradient $\nabla f^*$ maps $X^*$ to $X$ with
\begin{align}\label{Land.26}
\|\nabla f^*(\xi) -\nabla f^*(\eta) \| \le \frac{\|\xi-\eta\|}{2\sigma} 
\end{align}
for all $\xi, \eta \in X^*$. Consequently, it follows from (\ref{Land.25}) and (\ref{Land.26}) that
\begin{align}\label{Land.27}
D_f^\xi (\bar x,x) = f^*(\xi) - f^*(\bar \xi) - \l \xi - \bar \xi, \nabla f^*(\bar \xi)\r 
\le \frac{1}{4\sigma}\|\xi-\bar \xi\|^2
\end{align}
for any $(x, \xi), (\bar x, \bar \xi) \in \mbox{graph}(\p f)$.

\section{\bf Convergence rates of Landweber-type methods}\label{sect3}
\setcounter{equation}{0}

The convergence of general Landweber-type methods in Banach spaces, including (\ref{Land}) as a specific case, has been studied in \cite{Jin2016,JW2013}, particularly when the iteration is terminated using the discrepancy principle (\ref{DP0}). For the method (\ref{Land}), the convergence analysis relies on the following standard conditions on $\R$ and $F$.

\begin{Assumption}\label{ass0}
{\it $\R : X\to (-\infty, \infty]$ is a proper, lower semi-continuous, strongly convex function in 
the sense that there is a constant $\sigma>0$ such that
\begin{align*}
\R(t \bar x + (1-t) x) + \sigma t(1-t) \|\bar x -x\|^2 \le t \R(\bar x) + (1-t) \R(x)
\end{align*}
for all $\bar x, x\in \emph{dom}(\R)$ and $0\le t\le 1$.}
\end{Assumption}

\begin{Assumption}\label{ass1}
\begin{enumerate}[leftmargin = 0.9cm]
\item[\emph{(i)}] $X$ is a Banach space and $Y$ is a Hilbert space. 

\item[\emph{(ii)}] There exists $\rho>0$ such that $B_{2\rho}(x_0)\subset \emph{dom}(F)$ and (\ref{Land.min}) 
has a solution $x^\dag$ such that $D_\R^{\xi_0} (x^\dag, x_0) \le \frac{1}{2}\sigma \rho^2$.

\item[\emph{(iii)}] There is a constant $L$ such that $\|F'(x)\| \le L$ for all $x\in B_{2\rho}(x_0)$ and there exists $0\le \eta <1$ such that 
$$
\|F(\tilde x) - F(x) - F'(x) (\tilde x - x)\| \le \eta \|F(\tilde x) - F(x)\|
$$
for all $\tilde x, x \in B_{2\rho}(x_0)$.
\end{enumerate}
\end{Assumption}

Under Assumption \ref{ass0} and Assumption \ref{ass1}. it has been shown in \cite[Lemma 3.2]{JW2013} that 
$x^\dag$ is the unique solution of (\ref{Land.min}), and by using (\ref{Land.24}) we can conclude that 
$\|x^\dag - x_0\|^2 \le \rho^2/2<\rho^2$, i.e. $x^\dag \in B_\rho(x_0)$. Furthermore, we have the following 
convergence result. 

\begin{theorem}\label{Land.thm0}
Let Assumption \ref{ass0} and Assumption \ref{ass1} hold. Consider the Landweber-type method (\ref{Land}). 
Let $\tau >(1+\eta)/(1-\eta)$ be a given number and let $\{\gamma_k^\d\}$ be chosen by one of the following 
rules: 
\begin{enumerate}[leftmargin = 0.9cm]
\item[\emph{(i)}] $\gamma_k^\d = \gamma/L^2$ with $0<\gamma < 4\sigma(1-\eta - (1+\eta)/\tau)$; 

\item[\emph{(ii)}] $\gamma_k^\d = \min\left\{\frac{\gamma \|F(x_k^\d) - y^\d\|^2}{\|F'(x_k^\d)^*(F(x_k^\d)
- y^\d)\|^2}, \bar \gamma\right\}$ for some $\bar \gamma >0$ and $0< \gamma  < 4\sigma(1-\eta 
- (1+\eta)/\tau)$;

\item[\emph{(iii)}] $\gamma_k^\d = \min\left\{\frac{\gamma((1-\eta) \|F(x_k^\d) - y^\d\| - (1+\eta)\d) \|F(x_k^\d) - y^\d\|}{\|F'(x_k^\d)^*(F(x_k^\d) - y^\d)\|^2}, \bar \gamma\right\}$ for some $\bar \gamma >0$ and $0< \gamma < 4 \sigma$. 
\end{enumerate}
Then the discrepancy principle (\ref{DP0}) outputs a finite integer $k_\d$ and there is a solution $x^*$ of  (\ref{Land.eq}) in $B_{2\rho}(x_0) \cap \emph{dom}(\R)$ such that 
$$
\lim_{\d\to 0} \|x_{k_\d}^\d - x^*\| = 0 \quad \emph{and} \quad 
\lim_{\d \to 0} D_\R^{\xi_{k_\d}^\d}(x^*, x_{k_\d}^\d) = 0. 
$$
If in addition, $\emph{Ran}(F'(x)^*) \subset \overline{\emph{Ran}(F'(x^\dag)^*)}$ for all $x\in B_{2\rho}(x_0)$, then $x^* = x^\dag$. 
\end{theorem}

Theorem \ref{Land.thm0} with $\gamma_k^\d$ chosen by either (i) or (ii) has been proved in 
\cite{Jin2016,JW2013}, and the similar arguments there can be applied straightforwardly to 
establish the convergence result when $\gamma_k^\d$ is chosen by (iii). In the method (\ref{Land}),
the regularization functional $\R$ can be chosen in various ways to detect the features of sought 
solutions.

\begin{example}\label{ex1}
{\rm
Here are some choices of the regularization functional $\R$ that are useful in detecting features 
of sought solutions in applications. 

\begin{enumerate}[leftmargin = 0.9cm]
\item[(a)] When $X$ is a Hilbert space and the sought solution lies in a closed convex set $C$, 
we may take 
$$
\R(x) := \frac{1}{2} \|x\|^2 + \iota_C(x), 
$$
where $\iota_C$ denote the indicator function of $C$, i.e. $\iota_C(x) = 0$ if $x \in C$ and 
$\iota_C(x) = \infty$ otherwise. It is clear that this $\R$ satisfies Assumption \ref{ass0} 
with $\sigma = 1/2$. 

\item[(b)] When the sought solution lies in $L^2(\Omega)$ with sparsity feature, where 
$\Omega\subset {\mathbb R}^d$ is a bounded domain, we may take 
$$
\R(x) := \frac{1}{2} \int_\Omega |x(\omega)|^2 d \omega + \beta \int_\Omega |x(\omega)| d\omega, \quad 
x \in L^2(\Omega), 
$$
where $\beta>0$ is a large number. Clearly, this $\R$ satisfies 
Assumption \ref{ass0} with $\sigma = 1/2$.

\item[(c)] When the sought solution is piece-wise constant on some bounded domain 
$\Omega \subset {\mathbb R}^d$, we may use 
$$
\R(x) := \frac{1}{2} \int_\Omega |x(\omega)|^2 d\omega + \beta \int_\Omega |Dx|
$$
for some large number $\beta>0$, where $\int_\Omega|Dx|$ denotes the total variation of $x$ over $\Omega$ 
defined by (\cite{ABM2006})
$$
\int_\Omega |Dx| := \sup \left\{ \int_\Omega x \mbox{div} {\bf u} dx: {\bf u} \in C_0^1(\Omega, {\mathbb R}^d) 
\mbox{ and } \|{\bf u}\|_{L^\infty(\Omega)} \le 1\right\}.
$$
It is easy to show that this $\R$ satisfies Assumption \ref{ass0} with $\sigma = 1/2$. 

\item[(d)] When the sought solution $x^\dag$ is a probability density function over a bounded domain 
$\Omega \subset {\mathbb R}^d$, i.e. $x^\dag \ge 0$ a.e. on $\Omega$ and $\int_\Omega x = 1$, we may take 
$$
\R(x) := f(x) + \iota_\Delta(x),
$$
where $\iota_\Delta$ denotes the indicator function of 
$$
\Delta:= \left\{ x\in L_+^1(\Omega): \int_\Omega x^\dag = 1\right\}
$$
and 
$$
f(x) = \left\{\begin{array}{lll}
\int_\Omega x \log x, & \mbox{ if } x \in L_{+}^1(\Omega) \mbox{ and } x\log x \in L^1(\Omega), \\
+ \infty, & \mbox{ otherwise}
\end{array}\right. 
$$
is the negative of the Boltzmann-Shannon entropy. Here 
$L_{+}^1(\Omega) := \{x\in L^1(\Omega): x\ge 0 \mbox{ a.e. on } \Omega\}$. It is known that $\R$ 
satisfies Assumption \ref{ass0} with $\sigma = 1/2$; see \cite{BL1991,E1993,EL1993,Jin2022}. 
\end{enumerate}
}
\end{example}

The performance of the method (\ref{Land}) has been demonstrated in \cite{Jin2016,JW2013} through 
various numerical experiments. These results indicate that the method (\ref{Land}) not only captures 
the features of the sought solution with a suitable choice of $\R$, but also enjoys a certain convergence 
speed when the sought solutions exhibit desirable properties. This raises the natural question of whether 
it is possible to establish convergence rates for the method (\ref{Land}) when the sought solutions 
satisfy specific source conditions. In this section, we will derive the convergence rate for the method 
(\ref{Land}) under the assumption that the sought solution satisfies the benchmark source condition 
(\ref{sc0}), provided that $F$ meets the following additional condition.

\begin{Assumption}\label{ass1.5}
There exists $\kappa_0\ge 0$ such that for any $x \in B_{2\rho}(x_0)$ there is a bounded 
linear operator $Q_x: Y \to Y$ such that 
$$
F'(x) = Q_x F'(x^\dag)   \quad \mbox{and} \quad \|I - Q_x\| \le \kappa_0 \|x - x^\dag\|.
$$
\end{Assumption}

Assumption \ref{ass1.5} is actually stronger than condition (iii) in Assumption \ref{ass1}. Indeed, using  
similar arguments as in \cite{HNS1995,N2018} we can conclude that, if $8\kappa_0 \rho <1$, then  
$$
\|F(\tilde x) - F(x) - F'(x) (\tilde x - x)\|\le \eta \|F(\tilde x) - F(x)\|
$$
for all $\tilde x,x \in B_{2\rho}(x_0)$, where $\eta := 5\kappa_0 \rho/(1 - 8 \kappa_0 \rho)$. This implies 
that condition (iii) in Assumption \ref{ass1} holds with a small $\eta$, provided that $\kappa_0\rho$ is 
sufficiently small. 

It should be noted that when $F$ is a bounded linear operator, Assumption \ref{ass1.5} holds automatically 
with $\kappa_0 = 0$. Moreover, \cite{HNS1995,SEK1993} provide examples of nonlinear ill-posed inverse 
problems that satisfy Assumption \ref{ass1.5}, including those arising from nonlinear integral equations 
of the first kind and parameter identification in partial differential equations. Below we will provide one 
more example that satisfies Assumption \ref{ass1.5}. 

\begin{example}
{\rm
Let $\Omega\subset {\mathbb R}^d$ be a bounded domain with Lipschitz boundary $\p \Omega$. We consider the 
Neumann boundary value problem of the semi-linear elliptic equation
\begin{align}\label{NBP}
- \triangle y + g(y) = x \,\,\, \mbox{ in } \Omega, \qquad \p_\nu y = 0 \,\,\, \mbox{ on } \p \Omega, 
\end{align}
where $\p_\nu y$ denotes the normal derivative of $y$ in the direction of the unit outward normal $\nu$ to 
$\p \Omega$ and $g(y)$ is a given function defined on ${\mathbb R}$ satisfying the following properties:
\begin{enumerate}[leftmargin = 0.9cm]
\item[$\bullet$] $g \in C^1({\mathbb R})$ and there is a constant $\mu_0>0$ such that $g'(y) \ge \mu_0$ 
for all $y \in {\mathbb R}$.

\item[$\bullet$] $g'$ is locally Lipschitz continuous, i.e. for any $M>0$ there is $L_M\ge 0$ such that 
$$
|g'(y_1) - g'(y_2)| \le L_M|y_1 - y_2|
$$
for all $y_1, y_2\in {\mathbb R}$ satisfying $|y_1|, |y_2| \le M$. 
\end{enumerate}
Let $r>d/2$ be a fixed number. We consider the inverse problem of determining the source term 
$x\in L^r(\Omega)$ from an $L^2(\Omega)$-measurement of $y$. According to 
\cite[Theorem 4.7 \& Theorem 4.16]{T2010}, for each $x \in L^r(\Omega)$ problem (\ref{NBP}) has a unique 
weak solution $y_x \in H^1(\Omega) \cap C(\overline{\Omega})$ 
and for any $\tilde x, x\in L^r(\Omega)$ there holds
\begin{align}\label{NBP.2}
\|y_{\tilde x} - y_x\|_{H^1(\Omega)} + \|y_{\tilde x} - y_x\|_{C(\overline{\Omega})} 
\le C\|\tilde x - x\|_{L^r(\Omega)}, 
\end{align}
where $C$ denotes a generic constant depending only on $\mu_0$ and $\Omega$. Since 
$H^1(\Omega) \hookrightarrow L^2(\Omega)$, it makes sense to define the operator 
$F: L^r(\Omega) \to L^2(\Omega)$ by $F(x) := y_x$ for any $x \in L^r(\Omega)$ and thus our inverse 
problem reduces to the form (\ref{Land.eq}). 

According to \cite[Theorem 4.17]{T2010}, this $F$ is Fr\'{e}chet differentiable, and for any 
$x \in L^r(\Omega)$, its Fr\'{e}chet derivative $F'(x)$ is an operator from $L^r(\Omega)$ to 
$H^1(\Omega) \subset L^2(\Omega)$. For each $h \in L^r(\Omega)$, $v:= F'(x) h$ 
is the unique weak solution of 
\begin{align*}
-\triangle v + g'(y_x)v = h \,\,\, \mbox{ in } \Omega, \qquad 
\p_\nu v = 0 \,\,\, \mbox{ on } \p \Omega.
\end{align*}
Next, we show that this operator $F$ satisfies Assumption \ref{ass1.5}. To see this, we consider the 
linear subspace 
$$
S:= \{u\in H^1(\Omega): \p_\nu u =0 \mbox{ on } \p \Omega\}
$$
of $H^1(\Omega)$. Let $x_0\in L^r(\Omega)$ be fixed and let $\rho>0$ be a fixed number.  For any 
$\tilde x, x\in B_\rho(x_0) := \{z \in L^r(\Omega): \|z - x_0\|_{L^r(\Omega)} < \rho\}$, we define 
the linear operator $Q_{\tilde x,x}: S\subset H^1(\Omega) \to H^1(\Omega)\subset L^2(\Omega)$, where,
for each $v\in S$, $w:= Q_{\tilde x, x} v \in H^1(\Omega)$ is the unique weak solution of the linear 
elliptic problem 
\begin{align}\label{NBP.5}
-\triangle w + g'(y_{\tilde x}) w = -\triangle v + g'(y_{x}) v \,\,\, \mbox{ in } \Omega, \qquad
\p_\nu w = 0 \,\,\, \mbox{ on } \p\Omega. 
\end{align}
Since $\p_\nu v = 0$ on $\p\Omega$, we can write
\begin{align*}
\left\{\begin{array}{lll}
&-\triangle (w - v) + g'(y_{\tilde x}) (w - v) = ( g'(y_x) - g'(y_{\tilde x})) v \,\,\, \mbox{ in } \Omega, \\
&\p_\nu (w - v) = 0 \,\,\, \mbox{ on } \p\Omega. 
\end{array}\right.
\end{align*}
By the theory of linear elliptic equations, we conclude that
\begin{align*}
\|w-v\|_{H^1(\Omega)} \le C\| ( g'(y_x) - g'(y_{\tilde x})) v\|_{L^2(\Omega)}
\le  \|g'(y_x) - g'(y_{\tilde x})\|_{L^\infty(\Omega)} \|v\|_{L^2(\Omega)}.
\end{align*}
By virtue of (\ref{NBP.2}) and the local Lipschitz continuity of $g'$, there exists a constant $C_\rho$ 
depending on $\Omega$, $\mu_0$, $x_0$ and $\rho$ such that 
\begin{align*}
\|w-v\|_{L^2(\Omega)} \le \|w-v\|_{H^1(\Omega)} \le C_\rho \|\tilde x-x\|_{L^r(\Omega)} \|v\|_{L^2(\Omega)}, 
\quad \forall v \in S. 
\end{align*}
Since $C_0^\infty(\Omega) \subset S$, $S$ is dense in $L^2(\Omega)$.  Thus, the above inequality 
implies that $Q_{\tilde x, x}$ can be extended to a bounded linear operator from $L^2(\Omega)$ to itself and 
\begin{align}\label{NBP.6}
\|I - Q_{\tilde x, x}\|_{L^2(\Omega)\to L^2(\Omega)} \le C_\rho\|\tilde x - x\|_{L^r(\Omega)}.
\end{align}
Note that, for any $h \in L^r(\Omega)$ we have $v:=F'(x) h \in S$, and for the function $w$ defined by 
(\ref{NBP.5}) we have $w = F'(\tilde x) h$. Consequently
$$
F'(\tilde x) h = w = Q_{\tilde x, x}v = Q_{\tilde x, x} F'(x) h, \quad \forall h \in L^r(\Omega)
$$
which shows that $F'(\tilde x) = Q_{\tilde x, x} F'(x)$. This, together with (\ref{NBP.6}), shows that $F$ 
satisfies Assumption \ref{ass1.5}. 
}
\end{example}

\subsection {\bf Some key estimates}

Based on Assumptions \ref{ass0}--\ref{ass1.5}, our aim is to derive convergence rates of the 
Landweber-type method (\ref{Land}) under the benchmark source condition (\ref{sc0}) on the sought 
solution $x^\dag$, when the method is terminated by either an {\it a priori} stopping rule or the 
discrepancy principle. 

In this subsection we will establish some key estimates that will be used in the subsequent analysis. 
Note from the definition of $x_k^\d$ in (\ref{Land}) that $\xi_k^\d \in \p \R(x_k^\d)$. We first show that,
up to a preassigned stopping index, the iterates always stay in $B_{2\rho}(x_0)$ and 
the Bregman distances $\{D_\R^{\xi_k^\d}(x^\dag, x_k^\d)\}$ obey a useful recursive inequality.  

\begin{lemma}\label{Land:lem1}
Let Assumption \ref{ass0} and Assumption \ref{ass1} hold. Suppose $\gamma_k^\d$ is chosen such 
that 
\begin{align}\label{gamma}
\underline{\gamma} \le \gamma_k^\d 
\le \min\left\{\frac{\gamma\|F(x_k^\d)-y^\d\|^2}{\|F'(x_k^\d)^*(F(x_k^\d)-y^\d)\|^2}, \bar \gamma \right\}
\end{align}
for some positive constants $\underline{\gamma}$, $\gamma$ and $\bar \gamma$ with 
\begin{align}\label{gamma0}
c_0 := \frac{1}{2} \left(1-\eta - \frac{\gamma}{4\sigma}\right) >0. 
\end{align}
Let $\hat k_\d$ be an integer such that $c_1 \d^2 \hat k_\d < \frac{1}{2} \sigma \rho^2$, where 
$c_1 := (1+\eta)^2 \bar \gamma/(4c_0)$. Then  
$x_k^\d \in B_{2\rho}(x_0)$ for $0\le k \le \hat k_\d$ and 
\begin{align}\label{Land.28}
\Delta_{k+1}^\d - \Delta_k^\d \le - c_0 \gamma_k^\d \|F(x_k^\d) - y^\d\|^2 + c_1 \d^2
\end{align}
for all integers $0 \le k < \hat k_\d$, where $\Delta_k^\d:= D_\R^{\xi_k^\d}(x^\dag, x_k^\d)$.  
\end{lemma}

\begin{proof} 
We first show by induction that 
\begin{align}\label{Land.12}
x_k^\d \in B_{2\rho}(x_0) \quad \mbox{and} \quad 
\Delta_k^\d \le D_\R^{\xi_0}(x^\dag, x_0) + c_1 \d^2 k
\end{align}
for all integers $0\le k \le \hat k_\d$. It is trivial for $k = 0$ as $\xi_0^\d = \xi_0$ and 
$x_0^\d = x_0$. Next we assume that (\ref{Land.12}) holds for all $0\le k \le l$ for some 
$l < \hat k_\d$ and show that (\ref{Land.12}) still holds for $k = l + 1$. By the second equation in 
(\ref{Land}), we have $\xi_l^\d \in \p \R(x_l^\d)$. Thus we may use (\ref{Land.20}) and 
(\ref{Land.27}) to obtain 
\begin{align*}
\Delta_{l+1}^\d - \Delta_l^\d 
& = D_\R^{\xi_{l+1}}(x_l, x_{l+1}) + \left\l \xi_{l+1}^\d - \xi_l^\d, x_l^\d - x^\dag \right\r \\
& \le \frac{1}{4\sigma} \left\|\xi_{l+1}^\d - \xi_l^\d \right\|^2 
+ \left\l \xi_{l+1}^\d - \xi_l^\d, x_l^\d - x^\dag \right\r.
\end{align*}
By virtue of the first equation in (\ref{Land}), we further obtain 
\begin{align*}
\Delta_{l+1}^\d - \Delta_l^\d 
& \le \frac{1}{4\sigma} (\gamma_l^\d)^2\left\|F'(x_l^\d)^*(F(x_l^\d) - y^\d)\right\|^2 
- \gamma_l^\d \left\l F(x_l^\d) - y^\d, F'(x_l^\d) (x_l^\d - x^\dag)\right\r.
\end{align*}
By using $\|y^\d - y\| \le \d$ and (iii) of Assumption \ref{ass1}, we then have 
\begin{align}\label{Land.11}
\Delta_{l+1}^\d - \Delta_l^\d 
& \le \frac{1}{4\sigma} (\gamma_l^\d)^2 \left\|F'(x_l^\d)^*(F(x_l^\d) - y^\d)\right\|^2
- \gamma_l^\d \left\|F(x_l^\d) - y^\d\right\|^2 \displaybreak[0] \nonumber \\
& \quad \, - \gamma_l^\d \left\l F(x_l^\d) - y^\d, y^\d - y + y - F(x_l^\d) - F'(x_l^\d) (x^\dag - x_l^\d) \right\r
\displaybreak[0] \nonumber \\
& \le \frac{1}{4\sigma} (\gamma_l^\d)^2 \left\|F'(x_l^\d)^*(F(x_l^\d) - y^\d)\right\|^2
- \gamma_l^\d \left\|F(x_l^\d) - y^\d\right\|^2 \nonumber \\
& \quad \, + \gamma_l^\d \left\|F(x_l^\d) - y^\d\right\| \left((1+\eta)\d + \eta \|F(x_l^\d) -y^\d\|\right).
\end{align}
In view of (\ref{gamma}), we can obtain 
\begin{align*}
\Delta_{l+1}^\d - \Delta_l^\d 
& \le -\left(1 - \eta - \frac{\gamma}{4\sigma}\right) \gamma_l^\d \left\|F(x_l^\d) - y^\d\right\|^2 
+ (1+\eta) \gamma_l^\d \d \left\|F(x_l^\d) - y^\d\right\| \\
& = - 2 c_0 \gamma_l^\d \left\|F(x_l^\d) - y^\d\right\|^2 
+ (1+\eta) \gamma_l^\d \d \left\|F(x_l^\d) - y^\d\right\|.
\end{align*}
Combining this with the inequality 
$$
(1+\eta) \d \left\|F(x_l) - y^\d\right\|  
\le c_0 \left\|F(x_l^\d) - y^\d\right\|^2 + \frac{(1+\eta)^2}{4c_0} \d^2
$$
shows that
\begin{align}\label{Land.13}
\Delta_{l+1}^\d - \Delta_l^\d 
& \le - c_0 \gamma_l^\d \left\|F(x_l^\d) - y^\d\right\|^2 + \frac{(1+\eta)^2}{4 c_0} \gamma_l^\d \d^2 \nonumber \\
& \le - c_0 \gamma_l^\d \left\|F(x_l^\d) - y^\d\right\|^2 + c_1 \d^2,
\end{align}
where for the last step we used $\gamma_l^\d \le \bar \gamma$. By virtue of this inequality 
and the induction hypothesis, we have 
$$
\Delta_{l+1}^\d \le \Delta_l^\d + c_1 \d^2 
\le D_\R^{\xi_0}(x^\dag, x_0) + c_1 \d^2 (l+1) 
\le \frac{1}{2} \sigma \rho^2 + c_1 \d^2 \hat k_\d 
< \sigma \rho^2
$$
which together with Assumption \ref{ass0} and (\ref{Land.24}) implies that $\sigma \|x_{l+1}^\d - x^\dag\|^2 
\le \sigma \rho^2$ and hence $\|x_{l+1}^\d - x^\dag\| \le \rho$. Since Assumption \ref{ass0} and (ii) in 
Assumption \ref{ass1} imply $\|x^\dag - x_0\| \le \rho$, we thus have $\|x_{l+1}^\d - x_0\| \le 2\rho$, i.e.
$x_{l+1}^\d \in B_{2\rho}(x_0)$. We therefore complete the proof of (\ref{Land.12}). As a direct 
consequence, we can see that (\ref{Land.13}) holds for all $0\le l < \hat k_\d$ which shows the desired 
result. 
\end{proof}

In the following we will focus on deriving convergence rates under the benchmark source condition 
(\ref{sc0}). Let $A:=F'(x^\dag)$. The key idea is the observation that, under Assumption \ref{ass1.5}, 
$\xi_k^\d -\xi_0 \in \mbox{Ran}(A^*)$ for each integer $k$. This leads us to introducing an auxiliary 
sequence which plays a crucial role in our argument. To be more precise, by using Assumption \ref{ass1.5}
we can write the first equation in (\ref{Land}) as
$$
\xi_{k+1}^\d = \xi_k^\d - \gamma_k^\d A^* Q_{x_k^\d}^* \left(F(x_k^\d) - y^\d\right).
$$
Thus, if we define $(\la_k^\d, \xi_k^\d, x_k^\d)$ by setting $\la_0^\d = 0$ and 
\begin{align}\label{Land1}
\begin{split}
\xi_k^\d & = \xi_0 + A^*\la_k^\d, \\
x_k^\d & = \arg\min_{x\in X} \left\{\R(x) - \l \xi_k^\d, x\r\right\}, \\
\la_{k+1}^\d & = \la_k^\d - \gamma_k^\d Q_{x_k^\d}^*\left(F(x_k^\d) - y^\d\right), 
\end{split}
\end{align}
then these $(\xi_k^\d, x_k^\d)$ are the same as the ones produced by (\ref{Land}). Note that the definition
of $\la_{k+1}^\d$ relies on $Q_{x_k^\d}$ which requires the information of $x^\dag$. The reformulation 
(\ref{Land1}) of (\ref{Land}) is not for computational purpose, instead it will be used only for theoretical 
analysis and a proper use of the additional sequence $\{\la_k^\d\}$ will enable us to derive the convergence 
rate under the benchmark source condition (\ref{sc0}). 

Since $Y$ is a Hilbert space and $\{\la_k^\d\}$ is a sequence in $Y$, we may use the polarization identity
and the definition of $\la_{k+1}^\d$ to obtain  
\begin{align*}
& \left\|\la_{k+1}^\d - \la^\dag\right\|^2 - \left\|\la_k^\d - \la^\dag\right\|^2 \\
& = \left\|\la_{k+1}^\d - \la_k^\d\right\|^2 + 2 \left\l \la_{k+1}^\d - \la_k^\d, \la_k^\d - \la^\dag\right\r \\
& = (\gamma_k^\d)^2 \left\|Q_{x_k^\d}^*(F(x_k^\d) - y^\d)\right\|^2 
- 2\gamma_k^\d \left\l Q_{x_k^\d}^*(F(x_k^\d) - y^\d), \la_k^\d - \la^\dag\right\r \displaybreak[0]\\
& = (\gamma_k^\d)^2 \left\|Q_{x_k^\d}^*(F(x_k^\d) - y^\d)\right\|^2 
- 2\gamma_k^\d \left\l (Q_{x_k^\d}^* - I) (F(x_k^\d) - y^\d), \la_k^\d - \la^\dag\right\r \\
& \quad \, - 2\gamma_k^\d \left\l F(x_k^\d) - y - A(x_k^\d - x^\dag) + y - y^\d, \la_k^\d - \la^\dag\right\r \\
& \quad \, - 2 \gamma_k^\d \left\l A(x_k^\d - x^\dag), \la_k^\d - \la^\dag\right\r. 
\end{align*}
By the condition on $Q_x$, we have 
$$
\| I - Q_{x_k^\d}^*\| \le \kappa_0 \|x_k^\d - x^\dag\|.
$$
According to Lemma \ref{Land:lem1} and the strong convexity of $\R$, $\|x_k^\d - x^\dag\|$ is bounded
for $0\le k \le \hat k_\d$ and thus we can find a constant $c_2>0$ independent of $k$ and $\d$ such that 
$$
\|Q_{x_k^\d}^*\|^2 \le (1 + \kappa_0 \|x_k^\d - x^\dag\|)^2 \le c_2, \quad \forall 0\le k \le \hat k_\d. 
$$
Therefore 
\begin{align}\label{Land5}
& \left\|\la_{k+1}^\d - \la^\dag\right\|^2 - \left\|\la_k^\d - \la^\dag\right\|^2 \nonumber \\
& \le c_2 (\gamma_k^\d)^2 \left\|F(x_k^\d) - y^\d\right\|^2 
+ 2\kappa_0 \gamma_k^\d \left\|x_k^\d - x^\dag\right\| \left\|F(x_k^\d) - y^\d\right\| \left\|\la_k^\d - \la^\dag\right\|
\nonumber\\
& \quad \, + 2 \gamma_k^\d \d \left\|\la_k^\d - \la^\dag\right\|
+ 2\gamma_k^\d \left\| F(x_k^\d) - y - A(x_k^\d - x^\dag)\right\| \left\|\la_k^\d - \la^\dag\right\| \nonumber\\
& \quad \, - 2 \gamma_k^\d \left\l A^*\la_k^\d - A^*\la^\dag, x_k^\d - x^\dag\right\r. 
\end{align}
By using the source condition (\ref{sc0}), the relation between $\xi_k^\d$ and $\la_k^\d$, and the strong 
convexity of $\R$, we have 
\begin{align*}
\left\l A^* \la_k^\d - A^* \la^\dag, x_k^\d - x^\dag\right\r 
& = \left\l \xi_k^\d - \xi^\dag, x_k^\d - x^\dag\right\r \\
& = D_\R^{\xi_k^\d}(x^\dag, x_k^\d) + D_\R^{\xi^\dag} (x_k^\d, x^\dag) \\ 
& \ge \Delta_k^\d + \sigma \|x_k^\d - x^\dag\|^2.
\end{align*}
By virtue of Assumption \ref{ass1.5} and (iii) in Assumption \ref{ass1} we have 
\begin{align*}
\left\|F(x_k^\d) - y - A (x_k^\d - x^\dag)\right\|
& = \left\|\int_0^1 \left(F'(x^\dag + t(x_k^\d - x^\dag)) - A\right) (x_k^\d - x^\dag) dt\right\| \displaybreak[0]\\
& = \left\|\int_0^1 \left(Q_{x^\dag + t(x_k^\d - x^\dag)} - I\right) A (x_k^\d - x^\dag) dt\right\| \displaybreak[0]\\
& \le \int_0^1 \left\|Q_{x^\dag + t(x_k^\d - x^\dag)} - I\right\| \left\|A (x_k^\d - x^\dag)\right\| dt \displaybreak[0]\\
& \le \frac{1}{2} \kappa_0 \left\|x_k^\d - x^\dag\right\| \left\|A (x_k^\d - x^\dag)\right\| \displaybreak[0]\\
& \le \frac{1}{2} (1+\eta) \kappa_0 \left\|x_k^\d - x^\dag\right\| \left\|F(x_k^\d) -y\right\| \\
& \le \frac{1}{2} (1+\eta) \kappa_0 \left\|x_k^\d - x^\dag\right\| \left(\left\|F(x_k^\d) -y^\d\right\| + \d\right). 
\end{align*}
Combining these estimates with (\ref{Land5}) gives 
\begin{align*}
&\left\|\la_{k+1}^\d - \la^\dag\right\|^2 - \left\|\la_k^\d - \la^\dag\right\|^2 \nonumber \\
& \le c_2 (\gamma_k^\d)^2 \left\|F(x_k^\d) - y^\d\right\|^2 
+ (3+\eta) \kappa_0 \gamma_k^\d \left\|x_k^\d - x^\dag\right\| \left\|F(x_k^\d)-y^\d\right\| \left\|\la_k^\d - \la^\dag\right\| \nonumber\\
& \quad \, + (1+\eta) \kappa_0 \gamma_k^\d \d \left\|x_k^\d - x^\dag\right\| \left\|\la_k^\d - \la^\dag\right\| 
+ 2 \gamma_k^\d \d \left\|\la_k^\d - \la^\dag\right\| \\
& \quad \, - 2\sigma \gamma_k^\d \left\|x_k^\d - x^\dag\right\|^2 - 2\gamma_k^\d \Delta_k^\d. 
\end{align*}
By the boundedness of $\|x_k^\d - x^\dag\|$, we can find a constant $c_3$ independent of $k$ and $\d$ such that 
\begin{align*}
& \|\la_{k+1}^\d - \la^\dag\|^2 - \|\la_k^\d - \la^\dag\|^2 \nonumber \\
& \le c_2 (\gamma_k^\d)^2 \|F(x_k^\d) - y^\d\|^2 
+ (3+\eta) \kappa_0 \gamma_k^\d \|x_k^\d - x^\dag\| \|F(x_k^\d) - y^\d\| \|\la_k^\d - \la^\dag\| \nonumber\\
& \quad \, + c_3 \gamma_k^\d \d \|\la_k^\d - \la^\dag\| - 2\sigma \gamma_k^\d \|x_k^\d - x^\dag\|^2 
- 2\gamma_k^\d \Delta_k^\d
\end{align*}
for all $0\le k \le \hat k_\d$. By virtue of the Young's inequality we have
\begin{align*}
& (3+\eta) \kappa_0 \|x_k^\d - x^\dag\| \|F(x_k^\d) - y^\d\| \|\la_k^\d - \la^\dag\| \\
& \le 2\sigma \|x_k^\d - x^\dag\|^2 
+ \frac{(3+\eta)^2 \kappa_0^2}{8 \sigma} \|F(x_k^\d) - y^\d\|^2 \|\la_k^\d - \la^\dag\|^2,   
\end{align*}
we consequently obtain 
\begin{align*}
\|\la_{k+1}^\d - \la^\dag\|^2 
& \le \left(1+\frac{(3+\eta)^2 \kappa_0^2}{8\sigma} \gamma_k^\d\|F(x_k^\d)-y^\d\|^2\right) \|\la_k^\d - \la^\dag\|^2 \nonumber \\
& \quad \, + c_2 (\gamma_k^\d)^2 \|F(x_k^\d) - y^\d\|^2 + c_3 \gamma_k^\d \d \|\la_k^\d - \la^\dag\| -  2\gamma_k^\d \Delta_k^\d. 
\end{align*}
With the help of the condition $\gamma_k^\d \le \bar \gamma$, we thus obtain 
the following result.

\begin{lemma}\label{Land:lem2}
Let Assumption \ref{ass0}, Assumption \ref{ass1} and Assumption \ref{ass1.5} hold. Let the step-size 
$\gamma_k$ be chosen to satisfy (\ref{gamma}) and (\ref{gamma0}), and let $\hat k_\d$ be chosen as in 
Lemma \ref{Land:lem1}. If $x^\dag$ satisfies the source condition (\ref{sc0}), then 
\begin{align}\label{Land6}
\|\la_{k+1}^\d - \la^\dag\|^2 
& \le \left(1+\frac{(3+\eta)^2 \kappa_0^2}{8\sigma}\gamma_k^\d\|F(x_k^\d)-y^\d\|^2\right) 
\|\la_k^\d - \la^\dag\|^2 \nonumber \\
& \quad \, + c_2 \bar \gamma \gamma_k^\d \|F(x_k^\d) - y^\d\|^2 + c_3 \bar \gamma \d \|\la_k^\d - \la^\dag\| 
- 2\gamma_k^\d \Delta_k^\d 
\end{align}
for $0\le k <\hat k_\d$, where $c_2$ and $c_3$ are positive constants independent of $k$ and $\d$. 
\end{lemma}

Based on Lemma \ref{Land:lem2}, we will show under the source condition (\ref{sc0}) that $\la_k^\d$ is 
bounded for all $0\le k \le \hat k_\d$ if $\hat k_\d$ is chosen as $\hat k_\d := [c\d^{-1}]$ for some 
positive constant $c>0$, where, for any given number $t$, $[t]$ denotes the largest integer $\le t$. To this 
end, we need the following elementary result (\cite{JLZ2023}). 

\begin{lemma}\label{Land:lem3}
Let $\{a_k\}$ and $\{b_k\}$ be two sequences of nonnegative numbers such that
$$
a_k^2 \le b_k^2 + c \sum_{j=0}^{k-1} a_j, \quad k=0, 1,\cdots,
$$
where $c \ge 0$ is a constant. If $\{b_k\}$ is non-decreasing, then
$$
a_k \le b_k + c k, \quad k=0, 1, \cdots.
$$
\end{lemma}

\begin{lemma}\label{Land:lem4}
Let Assumption \ref{ass0}, Assumption \ref{ass1} and Assumption \ref{ass1.5} hold. Let the step-size 
$\gamma_k$ be chosen to satisfy (\ref{gamma}) and (\ref{gamma0}). If $x^\dag$ satisfies the source condition
(\ref{sc0}) and if $\hat k_\d$ is chosen as $\hat k_\d = [c \d^{-1}]$ for some positive constant $c$, then 
there exist $C>0$ and $\bar \d>0$ such that 
$$
\|\la_k^\d - \la^\dag\| \le C, \quad \forall 0\le k \le \hat k_\d \emph{ and } 0<\d \le \bar \d. 
$$
\end{lemma}

\begin{proof}
Note that $c_1 \d^2 \hat k_\d \le c_1 c \d$. Thus there is $\bar \d>0$ such that $c_1 \d^2 \hat k_\d 
\le \frac{1}{2} \sigma \rho^2$ for all $0< \d \le \bar \d$. Consequently, the result in Lemma \ref{Land:lem2} 
holds for $0\le k \le \hat k_\d$ and $0< \d \le \bar \d$. We set 
\begin{align*}
a_k &:= \|\la_k^\d - \la^\dag\|, \\
b_k &:= c_2 \bar \gamma \gamma_k^\d \|F(x_k^\d)-y^\d\|^2, \\
\beta_k &:= \frac{(3+\eta)^2 \kappa_0^2 }{8\sigma} \gamma_k^\d \|F(x_k^\d)-y^\d\|^2.
\end{align*}
Then from (\ref{Land6}) in Lemma \ref{Land:lem2} it follows that  
\begin{align*}
a_{k+1}^2 \le (1+ \beta_k) a_k^2 + b_k + c_3 \bar \gamma \d a_k, \quad \forall 0\le k < \hat k_\d. 
\end{align*}
By recursively using this inequality we have 
\begin{align*}
a_k^2 \le a_0^2 \prod_{l=0}^{k-1} (1+ \beta_l) 
+ \sum_{i=0}^{k-1} \left(b_i + c_3 \bar \gamma \d a_i\right) \prod_{l=i+1}^{k-1} (1 + \beta_l).
\end{align*}
Note that 
\begin{align*}
\prod_{l=i+1}^{k-1} (1+\beta_l) 
&= \exp\left(\log \prod_{l=i+1}^{k-1} (1+\beta_l)\right) \\
& = \exp\left(\sum_{l=i+1}^{k-1} \log(1+\beta_l)\right) \\
& \le \exp\left(\sum_{l=i+1}^{k-1} \beta_l\right), 
\end{align*}
where for the last step we used the inequality $\log(1+t) \le t$ for $t \ge 0$. Therefore 
\begin{align*}
a_k^2 \le a_0^2 \exp\left(\sum_{l=0}^{k-1} \beta_l\right) 
+ \sum_{i=0}^{k-1} \left(b_i + c_3 \bar \gamma \d a_i\right) \exp\left(\sum_{l=i+1}^{k-1} \beta_l\right).
\end{align*}
According to Lemma \ref{Land:lem1}, we have 
\begin{align}\label{Land10}
c_0 \sum_{l=0}^{k-1} \gamma_l^\d \|F(x_l^\d) - y^\d\|^2 
\le \Delta_0^\d + c_1 \d^2 k \le D_\R^{\xi_0}(x^\dag, x_0) + c_1 \d^2 \hat k_\d.
\end{align}
Thus, there is a constant $\tilde C$ such that 
$$
\exp\left(\sum_{l=0}^{k-1} \beta_l\right) 
= \exp\left(\frac{(3+\eta)^2 \kappa_0^2}{8\sigma} \sum_{l=0}^{k-1} \gamma_l^\d \|F(x_l) - y^\d\|^2\right) 
\le \tilde C 
$$
for all $0\le k \le \hat k_\d$. Consequently  
\begin{align*}
a_k^2 \le \tilde C \left(a_0^2 + \sum_{i=0}^{k-1} b_i\right) 
+ \tilde C c_3 \bar \gamma \d \sum_{i=0}^{k-1} a_i, \quad \forall 0\le k \le \hat k_\d. 
\end{align*}
By virtue of Lemma \ref{Land:lem3} then we can conclude that  
\begin{align*}
\|\la_k^\d - \la^\dag\| 
& = a_k \le \sqrt{\tilde C \left(a_0^2 + \sum_{i=0}^{k-1} b_i\right)} + \tilde C c_3 \bar \gamma \d k  \\
& = \sqrt{\tilde C \left(a_0^2 + c_2 \bar \gamma \sum_{i=0}^{k-1} \gamma_i^\d \|F(x_i^\d)-y^\d\|^2\right)} 
+ \tilde C c_3 \bar \gamma \d k.
\end{align*}
With the help of (\ref{Land10}), we can find a positive constant $C'$ such that  
\begin{align*}
\|\la_k^\d - \la^\dag\| \le C' (1 + k \d), \quad \forall 0\le k \le \hat k_\d \mbox{ and } 0< \d \le \bar \d. 
\end{align*}
Since $\hat k_\d = [c\d^{-1}]$, we therefore complete the proof. 
\end{proof}

\subsection{\bf Convergence rate under {\it a priori} stopping rule}

The following result gives the convergence rate of the method (\ref{Land}) under an {\it a prior} stopping 
rule when the source condition (\ref{sc0}) holds and the step-size $\gamma_k^\d$ is chosen properly. 

\begin{theorem}\label{Land:thm1}
Let Assumption \ref{ass0}, Assumption \ref{ass1} and Assumption \ref{ass1.5} hold. Let the step-size 
$\gamma_k^\d$ be chosen to satisfy (\ref{gamma}) and (\ref{gamma0}). If $x^\dag$ satisfies the source 
condition (\ref{sc0}) and if $\hat k_\d$ is chosen as $\hat k_\d = [c \d^{-1}]$ for some positive constant 
$c$, then there exist $C>0$ and $\bar \d>0$ such that 
$$
D_\R^{\xi_{\hat k_\d}^\d}(x^\dag, x_{\hat k_\d}^\d) \le C \d, \quad \forall 0<\d \le \bar \d. 
$$
Consequently $\|x_{\hat k_\d}^\d - x^\dag\| = O(\d^{1/2})$ by the strong convexity of $\R$. 
\end{theorem}

\begin{proof}
By virtue of Lemma \ref{Land:lem2} and Lemma \ref{Land:lem4}, there is a positive constant $C$ such that  
\begin{align}\label{Land8}
\|\la_{k+1}^\d - \la^\dag\|^2 
& \le \|\la_k^\d - \la^\dag\|^2 + C \gamma_k^\d \|F(x_k^\d) - y^\d\|^2 + C \d - 2 \gamma_k^\d \Delta_k^\d. 
\end{align}
for all integers $0 \le k < \hat k_\d$ and $0<\d \le \bar \d$. Multiplying this equation by $c_0/C$, adding 
to the equation (\ref{Land.28}) in Lemma \ref{Land:lem1}, and using $\gamma_k^\d \ge \underline{\gamma}$, 
we have
\begin{align*}
\Delta_{k+1}^\d + \frac{c_0}{C} \|\la_{k+1}^\d - \la^\dag\|^2 
\le \Delta_k^\d + \frac{c_0}{C} \|\la_k^\d - \la^\dag\|^2 + c_0 \d + c_1 \d^2 - c_4 \Delta_k^\d
\end{align*}
for all $0\le k < \hat k_\d$, where $c_4 := 2c_0 \underline{\gamma}/C$. Recursively using this 
inequality gives 
\begin{align}\label{Land9}
\Delta_{\hat k_\d}^\d + \frac{c_0}{C} \left\|\la_{\hat k_\d}^\d - \la^\dag\right\|^2 
+ c_4 \sum_{k=0}^{\hat k_\d-1} \Delta_k^\d
\le \Delta_0 + \frac{c_0}{C} \|\la_0 - \la^\dag\|^2 + (c_0 + c_1 \d) \d \hat k_\d,
\end{align}
where $\Delta_0:= \Delta_0^\d = D_\R^{\xi_0}(x^\dag, x_0)$. According to Lemma \ref{Land:lem1} we have 
$$
\Delta_k^\d \ge \Delta_{\hat k_\d}^\d - c_1 (\hat k_\d - k) \d^2, \quad 0\le k \le \hat k_\d
$$
and thus 
\begin{align*}
\sum_{k=0}^{\hat k_\d-1} \Delta_k^\d \ge \hat k_\d \Delta_{\hat k_\d}^\d 
- c_1 \d^2 \sum_{k=0}^{\hat k_\d-1} (\hat k_\d - k)
= \hat k_\d \Delta_{\hat k_\d}^\d - \frac{1}{2} c_1 \hat k_\d (\hat k_\d + 1) \d^2.
\end{align*}
Combining this with (\ref{Land9}) gives 
\begin{align*}
\min\{1, c_4\} (\hat k_\d + 1) \Delta_{\hat k_\d}^\d
\le \Delta_0 + \frac{c_0}{C} \|\la^\dag\|^2 + (c_0 + c_1 \d) \d \hat k_\d
+ \frac{1}{2} c_1 c_4 \hat k_\d (\hat k_\d+1) \d^2. 
\end{align*}
Therefore, there is a constant $C'$ independent of $\d$ such that 
\begin{align*}
\Delta_{\hat k_\d}^\d \le C' \left(\frac{1}{\hat k_\d + 1} + \hat k_\d \d^2 + \d\right), \quad 
\forall 0< \d \le \bar \d. 
\end{align*}
According to the choice of $\hat k_\d$, we thus obtain $\Delta_{\hat k_\d}^\d \le C \d$.  Finally, by the 
strong convexity of $\R$, we have $\|x_{\hat k_\d}^\d - x^\dag\|^2 \le \Delta_{\hat k_\d}^\d/\sigma = O(\d)$. 
The proof is complete. 
\end{proof}

\begin{remark}\label{Land:rk1}
{\rm 
Theorem \ref{Land:thm1} requires the step-size $\gamma_k^\d$ to satisfy (\ref{gamma}) and (\ref{gamma0}). 
Actually there are various choices of $\gamma_k^\d$ satisfying this requirement:

\begin{enumerate}[leftmargin = 0.9cm]
\item[(i)] If $\gamma_k^\d$ is taken to be a constant step-size $\gamma_k^\d = \gamma/L^2$ with 
$0< \gamma < 4\sigma (1-\eta)$, where $L$ is the finite constant appeared in (iii) of Assumption \ref{ass1}, 
then (\ref{gamma}) and (\ref{gamma0}) are satisfied with $\underline{\gamma} = \bar\gamma = \gamma/L^2$. 

\item[(ii)] If $\gamma_k^\d$ is chosen as 
$$
\gamma_k^\d = \min\left\{\frac{\gamma \|F(x_k^\d)-y^\d\|^2}{\|F'(x_k^\d)^*(F(x_k^\d)-y^\d)\|^2}, \bar \gamma\right\}
$$
for some constants $\bar \gamma > 0$ and $0 < \gamma < 4 \sigma (1-\eta)$, then (\ref{gamma}) and 
(\ref{gamma0}) are satisfied with $\underline{\gamma} = \min\{\gamma/L^2, \bar \gamma\}$. This choice 
of $\gamma_k^\d$ is related to the minimal error method.

\item[(iii)] Let $\tau > (1+\eta)/(1-\eta)$ be a given number and set $r_k^\d:= F(x_k^\d)-y^\d$. If 
$\gamma_k^\d$ is chosen as 
\begin{align*}
\gamma_k^\d = \left\{\begin{array}{lll}
\min\left\{\frac{\gamma_0 ((1-\eta)\|r_k^\d\|-(1+\eta) \d) \|r_k^\d\|}
{\|F'(x_k^\d)^* r_k^\d\|^2}, \bar \gamma\right\} 
& \mbox{ if } \|r_k^\d\| > \tau \d, \\[1.2ex]
\min\left\{\frac{\gamma_0 (1-\eta)}{L^2}, \bar \gamma\right\}, & \mbox{ otherwise}
\end{array}\right.
\end{align*}
for some constant $0<\gamma_0< 4\sigma$, then by noting that 
\begin{align*}
\frac{\gamma_0 ((1-\eta)\|r_k^\d\|-(1+\eta) \d) \|r_k^\d\|}
{\|F'(x_k^\d)^*r_k^\d\|^2} 
& \ge \gamma_0 \left(1-\eta - \frac{1+\eta}{\tau}\right) \frac{\|r_k^\d\|^2}
{\|F'(x_k^\d)^* r_k^\d\|^2} \\
& \ge \frac{\gamma_0}{L^2} \left(1-\eta - \frac{1+\eta}{\tau}\right)
\end{align*}
when $\|r_k^\d\| > \tau \d$, we can see that (\ref{gamma}) and (\ref{gamma0}) are satisfied with 
$$
\gamma = \gamma_0 (1-\eta) \quad \mbox{ and } \quad 
\underline{\gamma} = \min\left\{\frac{\gamma_0}{L^2} \left(1-\eta - \frac{1+\eta}{\tau}\right), \bar \gamma\right\}.
$$
\end{enumerate}
Therefore, the result in Theorem \ref{Land:thm1} holds for all the above choices of step-sizes, i.e. 
if the source condition (\ref{sc0}) holds and $\hat k_\d = [c\d^{-1}]$, then $\Delta_{\hat k_\d}^\d =O(\d)$
and $\|x_{k_\d}^\d - x^\dag\| = O(\d^{1/2})$. 
}
\end{remark}

\begin{remark}
{\rm 
When $F$ is a bounded linear operator, $\R(x) = \frac{1}{2}\|x\|^2$ and $\gamma_k^\d$ is chosen as 
$$
\gamma_k^\d = \frac{\gamma\|F x_k^\d - y^\d\|^2}{\|F^*(F x_k^\d - y^\d)\|^2},
$$
the corresponding method of (\ref{Land}) becomes the minimal error method for solving linear 
ill-posed problems. It has been shown in \cite{ELP1990} that minimal error method is not a regularization 
method when terminated by an {\it a priori} stopping rule. Our result shows that if this step-size 
is modified into the one as listed in (ii) of Remark \ref{Land:rk1}, convergence rate can be derived under 
an {\it a priori} stopping rule if the source condition (\ref{sc0}) holds. Our result does not contradict
the one in \cite{ELP1990} because we have imposed an upper bound on $\gamma_k^\d$. Placing such an upper 
bound can enhance stability in numerical computation. 
}
\end{remark}

\subsection{\bf Convergence rate under the discrepancy principle}

Next we turn to derive the convergence rate of the Landweber-type method (\ref{Land}) under the source 
condition (\ref{sc0}) when the iteration is terminated by the discrepancy principle (\ref{DP0})
for some number $\tau >1$. We need the following result. 

\begin{lemma}\label{Land:lem5}
Let Assumption \ref{ass0} and Assumption \ref{ass1} hold. Let the step-size $\gamma_k^\d$ be chosen to 
satisfy (\ref{gamma}). If $\tau>1$ and $\gamma>0$ are chosen such that 
\begin{align}\label{gamma2}
\tau > \frac{1+\eta}{1-\eta} \quad \mbox{and} \quad \gamma < 4 \sigma\left(1-\eta - \frac{1+\eta}{\tau}\right),
\end{align}
then the discrepancy principle (\ref{DP0}) outputs a finite integer $k_\d$, $x_k^\d \in B_{2\rho}(x_0)$
for all $0\le k \le k_\d$, and 
\begin{align*}
\Delta_{k+1}^\d - \Delta_k^\d \le - c_5 \gamma_k^\d \|F(x_k^\d) - y^\d\|^2 
\end{align*}
for all $0\le k < k_\d$, where $\Delta_k^\d := D_\R^{\xi_k^\d}(x^\dag, x_k^\d)$ and 
$c_5 := 1 - \eta - \frac{1+\eta}{\tau} - \frac{\gamma}{4\sigma} >0$. 
\end{lemma}

\begin{proof}
This can be done by essentially following the proof in \cite[Lemma 3.4]{JW2013} or \cite[Lemma 3.1]{Jin2016}. 
\end{proof}

\begin{theorem}\label{Land:thm2}
Let Assumption \ref{ass0}, Assumption \ref{ass1} and Assumption \ref{ass1.5} hold. Let the step-size 
$\gamma_k^\d$ be chosen to satisfy (\ref{gamma}). Assume that $\tau>1$ and $\gamma>0$ are chosen such that 
(\ref{gamma2}) holds and let $k_\d$ be the integer determined by the discrepancy principle (\ref{DP0}). If 
the source condition (\ref{sc0}) holds, then there is a positive constant $C$ such that 
\begin{align}\label{Land.17}
\Delta_{k_\d}^\d \le C \d \quad \mbox{and} \quad 
\|x_{k_\d}^\d - x^\dag\| \le C\d^{1/2}
\end{align}
for all $\d>0$, where $\Delta_{k_\d}^\d := D_\R^{\xi_{k_\d}^\d} (x^\dag, x_{k_\d}^\d)$.
\end{theorem}

\begin{proof}
By the strong convexity of $\R$, the second estimate in (\ref{Land.17}) follows from the first one.
Thus we only need to show the first estimate in (\ref{Land.17}).

Since $\tau>1$ and $\gamma>0$ satisfy (\ref{gamma2}), the step-size $\gamma_k^\d$ satisfies (\ref{gamma}) 
and (\ref{gamma0}). Thus, by taking $\hat k_\d := [\d^{-1}]$, we may use Theorem \ref{Land:thm1} to 
conclude that there exists constants $C'>0$ and $\bar \d>0$ such that  
\begin{align}\label{Land.29}
D_\R^{\xi_{\hat k_\d}^\d} (x^\dag, x_{\hat k_\d}^\d) \le C' \d, \quad \forall 0< \d \le \bar \d. 
\end{align}
We will use this result to show the first estimate in (\ref{Land.17}). According to Lemma \ref{Land:lem5},
$\Delta_{k_\d}^\d \le D_\R^{\xi_0}(x^\dag, x_0) < \infty$. Thus, we can find a positive constant $C$ 
such that $\Delta_{k_\d}^\d \le C \d$ for $\d>\bar \d$. In the following we assume $0<\d \le \bar \d$. 

We consider $k_\d$ according to two cases. If $k_\d \ge \hat k_\d$, then we may use Lemma \ref{Land:lem5} 
and (\ref{Land.29}) to obtain  
\begin{align*}
\Delta_{k_\d}^\d \le D_\R^{\xi_{\hat k_\d}^\d} (x^\dag, x_{\hat k_\d}^\d) \le C' \d.
\end{align*}
It remains only to consider the case $k_\d < \hat k_\d$. According to Lemma \ref{Land:lem4}, there 
is a constant $C$ independent of $\d$ such that $\|\la_{k_\d}^\d - \la^\dag\|\le C$. Thus, by using 
$\xi_{k_\d}^\d = \xi_0 + A^* \la_{k_\d}^\d$ and the source condition (\ref{sc0}), we can obtain 
\begin{align*}
\Delta_{k_\d}^\d  
& \le D_{\R}^{\xi_{k_\d}^\d} (x^\dag, x_{k_\d}^\d) + D_{\R}^{\xi^\dag} (x_{k_\d}^\d, x^\dag) 
= \l \xi_{k_\d}^\d - \xi^\dag, x_{k_\d}^\d - x^\dag\r \\
& = \l A^* \la_{k_\d}^\d - A^* \la^\dag, x_{k_\d}^\d - x^\dag\r  
= \l \la_{k_\d}^\d - \la^\dag, A(x_{k_\d}^\d - x^\dag)\r \\
& \le \|\la_{k_\d}^\d - \la^\dag\| \|A(x_{k_\d}^\d - x^\dag)\| \\
& \le C \|A (x_{k_\d}^\d - x^\dag)\|.
\end{align*}
By Assumption \ref{ass1} (iii), $\|y^\d - y\| \le \d$, and $\|F(x_{k_\d}^\d) - y^\d\| \le \tau \d$,
we can conclude that
\begin{align*}
\Delta_{k_\d}^\d  
& \le C (1+\eta)\|F(x_{k_\d}^\d) - y\|\\
& \le C (1+\eta)\left(\d + \|F(x_{k_\d}^\d) - y^\d\|\right) \\
& \le C (1+\eta)(1 + \tau) \d. 
\end{align*}
The proof is therefore complete. 
\end{proof}

By applying the above result to the method (\ref{Land}) with the step-size $\gamma_k^\d$ chosen by those 
rules listed as (i) and (ii) in Remark \ref{Land:rk1}, we immediately obtain the following result.

\begin{corollary}\label{Land:cor1}
Let Assumption \ref{ass0}, Assumption \ref{ass1} and Assumption \ref{ass1.5} hold. Consider the method 
(\ref{Land}) with $\gamma_k^\d$ chosen by one of the following rules:
\begin{enumerate}[leftmargin = 0.9cm]
\item[\emph{(i)}] $\gamma_k^\d = \gamma/L^2$ for some $\gamma>0$;

\item[\emph{(ii)}] $\gamma_k^\d = \min\left\{\frac{\gamma \|F(x_k^\d) - y^\d\|^2}{\|F'(x_k^\d)^*(F(x_k^\d)-y^\d)\|^2}, 
\bar \gamma\right\}$ for some $\gamma>0$ and $\bar \gamma>0$.
\end{enumerate}
Assume $\tau>1$ and $\gamma > 0$ are chosen such that (\ref{gamma2}) holds and let $k_\d$ be the 
integer determined by the discrepancy principle (\ref{DP0}). If the source condition (\ref{sc0}) 
holds, then there is a constant $C>0$ such that
\begin{align*}
D_\R^{\xi_{k_\d}^\d}(x^\dag, x_{k_\d}^\d) \le C \d \quad \mbox{ and } \quad 
\|x_{k_\d}^\d - x^\dag\| \le C \d^{1/2}
\end{align*}
for all $\d >0$. 
\end{corollary}

In Theorem \ref{Land:thm2} and Corollary \ref{Land:cor1}, $\tau>1$ and $\gamma>0$ are required to satisfy 
(\ref{gamma2}). 
In applications we may need to take $\tau$ close to $(1+\eta)/(1-\eta)$ in order to get more 
accurate reconstruction results. As a direct consequence, the theoretical result only allows to 
use very small $\gamma$. This may lead to increase the required number of iterations and thus
consume more computational time. For the step-size $\gamma_k^\d$ given by (iii) in Remark \ref{Land:rk1},
fortunately we have the following result which allows $\gamma_0$ to be a much larger number. 

\begin{theorem}\label{Land:thm3}
Let Assumption \ref{ass0}, Assumption \ref{ass1} and Assumption \ref{ass1.5} hold. Let 
$\tau>(1+\eta)/(1-\eta)$ and let the step-size $\gamma_k^\d$ be chosen as
\begin{align}\label{gamma3}
\gamma_k^\d = \min\left\{\frac{\gamma_0((1-\eta) \|r_k^\d\|-(1+\eta) \d)\|r_k^\d\|}{\|F'(x_k^\d)^* r_k^\d\|^2}, \bar \gamma\right\}
\end{align}
with $\bar\gamma>0$ and $0<\gamma_0<4\sigma$ whenever $\|r_k^\d\| >\tau \d$, where $r_k^\d:= F(x_k^\d)-y^\d$. 
Then the discrepancy principle (\ref{DP0}) outputs a finite integer $k_\d$. If the source condition 
(\ref{sc0}) holds, then there is a constant $C>0$ such that
\begin{align*}
D_\R^{\xi_{k_\d}^\d} (x^\dag, x_{k_\d}^\d) \le C \d \quad \mbox{and} \quad 
\|x_{k_\d}^\d - x^\dag\| \le C\d^{1/2}
\end{align*}
for all $\d>0$. 
\end{theorem}

\begin{proof}
The key point is to show that the same result in Lemma \ref{Land:lem5} holds under the relaxed 
requirement on $\gamma_0$. To see this, we note that, when $x_k^\d \in B_{2\rho}(x_0)$ with 
$\|r_k^\d\| >\tau\d$ and $\gamma_k^\d$ is chosen by (\ref{gamma3}), we may follow the derivation 
of (\ref{Land.11}) to obtain 
\begin{align*}
\Delta_{k+1}^\d - \Delta_k^\d
& \le \frac{1}{4\sigma} (\gamma_k^\d)^2 \|F'(x_k^\d)^* r_k^\d\|^2 
- \left((1-\eta) \|r_k^\d\| - (1+\eta) \d\right) \gamma_k^\d \|r_k^\d\| \\
& \le - \left(1- \frac{\gamma_0}{4\sigma}\right) \left((1-\eta) \|r_k^\d\| - (1+\eta) \d\right) \gamma_k^\d \|r_k^\d\| \\
& \le - \left(1- \frac{\gamma_0}{4\sigma}\right) \left(1-\eta - \frac{1+\eta}{\tau} \right) \gamma_k^\d \|r_k^\d\|^2. 
\end{align*}
Recall that $0< \gamma_0 < 4\sigma$ and $\tau > (1+\eta)/(1-\eta)$. Based on the above inequality, we 
may use a similar induction argument in the proof of Lemma \ref{Land:lem1} to show that, if $n \ge 0$ 
is an integer such that $\|r_k^\d\| >\tau \d$ for all $0\le k \le n$, then  $x_k^\d \in B_{2\rho}(x_0)$ 
for all $0\le k \le n$ and  
\begin{align}\label{Land.14}
\Delta_{k+1}^\d \le \Delta_k^\d - c_5 \|r_k^\d\|^2, \quad \forall 0\le k < n,
\end{align}
where
$$
c_5 := \left(1- \frac{\gamma_0}{4\sigma}\right) \left(1-\eta - \frac{1+\eta}{\tau} \right)
\min\left\{\frac{\gamma_0}{L^2} \left(1-\eta - \frac{1+\eta}{\tau}\right), \bar \gamma\right\} > 0. 
$$
If the discrepancy principle (\ref{DP0}) does not output a finite integer, then $\|r_k^\d\| > \tau \d$ 
for all integers $k \ge 0$. Consequently, it follows from (\ref{Land.14}) that 
\begin{align*}
c_5 \tau^2 \d^2 (n+1) \le c_5 \sum_{k=0}^n \|r_k^\d\|^2 \le \Delta_0^\d = D_\R^{\xi_0}(x^\dag, x_0) <\infty
\end{align*}
for any integer $n \ge 0$. Letting $n \to \infty$ gives a contradiction. Therefore, the discrepancy 
principle (\ref{DP0}) determines a finite integer $k_\d$  with 
\begin{align}\label{Land.16}
x_k^\d \in B_{2\rho}(x_0) \mbox{ for } 0\le k \le k_\d \quad \mbox{ and } \quad 
\Delta_{k+1}^\d \le \Delta_k^\d \mbox{ for } 0\le k < k_\d. 
\end{align}

Finally let us derive the given convergence rate. To this end, we define $\gamma_k^\d$ as in (iii) 
of Remark \ref{Land:rk1} for $k >k_\d$. For this choice of step-size, we know that the result in 
Theorem \ref{Land:thm1} holds, i.e. for $\hat k_\d = [\d^{-1}]$ there exist positive constants $C>0$ 
and $\bar \d>0$ such that 
$$
D_\R^{\xi_{\hat k_\d}^\d} (x^\dag, x_{\hat k_\d}^\d) \le C \d, \quad \forall 0<\d \le \bar \d. 
$$
Based on this inequality and (\ref{Land.16}), we may use the same argument in the proof of Theorem 
\ref{Land:thm2} to derive that $\Delta_{k_\d}^\d \le C \d$ for all $\d>0$. The proof is complete. 
\end{proof}

\begin{remark}
{\rm
(i) When $X$ is a Hilbert space, $\R(x) = \frac{1}{2} \|x\|^2$ and $\gamma_k^\d = \mbox{constant}$, the 
method (\ref{Land}) becomes the classical Landweber iteration (\ref{cLand}) studied in \cite{HNS1995} 
in which the convergence rate result under the source condition (\ref{sc0}) requires $\|\la^\dag\|$ to 
be sufficiently small. It is interesting to see that our result does not require this smallness condition. 

(ii) For the method (\ref{Land}) with a constant step-size and $F$ being a bounded linear operator, 
convergence rates have been established in \cite{Jin2022} under general variational source conditions by 
interpreting the method as a dual gradient method. Our work in this paper, restricted to this special 
case, can even present new results by allowing the use of variable step-sizes. 
}
\end{remark}

\section{\bf Convergence rate of stochastic mirror descent method}\label{sect4}
\setcounter{equation}{0}

In this section we will extend the idea to deal with a stochastic version of the method (\ref{Land}) for solving nonlinear ill-posed system 
\begin{align}\label{sys}
F_i(x) = y_i, \quad i = 1, \cdots, N
\end{align}
with exact data, where, for each $i = 1, \cdots, N$, the operator $F_i: X\to Y_i$ is a Fr\'{e}chet 
differentiable operator from a common Banach space $X$ to a Hilbert space $Y_i$. We assume (\ref{sys}) has a 
solution and look for a solution whose feature is described by a proper, lower semi-continuous, strongly 
convex function $\R: X \to (-\infty, \infty]$. Let $(x_0, \xi_0) \in X \times X^*$ be an initial guess 
with $x_0 \in X$ and $\xi_0 \in \p \R(x_0)$, we intend to find a solution $x^\dag$ of 
(\ref{sys}) such that 
\begin{align}\label{sys.min}
D_\R^{\xi_0}(x^\dag, x_0) = \min\left\{D_\R^{\xi_0}(x, x_0) : F_i(x) = y_i, \ i=1, \cdots, N\right\}.
\end{align}
Let $Y_1\times \cdots \times Y_N$ be the product space of $Y_1, \cdots, Y_N$ with the natural inner product 
inherited from those of $Y_i$. By introducing the operator 
$F: X \to Y_1 \times \cdots \times Y_N$ as 
\begin{align*}
F(x):= (F_1(x), \cdots, F_N(x))
\end{align*}
and $y := (y_1, \cdots, y_N)$, we can write (\ref{sys}) and (\ref{sys.min}) into the form (\ref{Land.eq}) 
and (\ref{Land.min}) respectively. In this way, we may use the method (\ref{Land}) to solve the problem. 
Correspondingly, the updating formula from $\xi_k$ to $\xi_{k+1}$ takes the form 
\begin{align*}
\xi_{k+1} = \xi_k - \gamma_k \sum_{i=1}^N F_i'(x_k)^*(F_i(x_k) - y_i)
\end{align*}
which requires calculating the terms $F_i'(x_k)^*(F_i(x_k) - y_i)$ for all $i = 1, \cdots, N$ at each 
iteration step and hence it can be time-consuming when $N$ is large. In order to make the method more 
efficient, we may take a random term from the sum in the above equation to update $\xi_{k+1}$. This leads 
to the following stochastic mirror descent method.

\begin{algorithm}\label{alg:SMD}
Take $(x_0, \xi_0) \in X \times X^*$ with $x_0 \in \emph{dom}(F)$ and $\xi_0 \in \p \R(x_0)$. Let $\{\gamma_k\}$ be a given sequence of positive numbers. For $k \ge 0$ do the following:
\begin{enumerate}[leftmargin = 0.9cm]
\item[\emph{(i)}] Pick $i_k \in \{1, \cdots, N\}$ randomly via the uniform distribution;

\item[\emph{(ii)}] Update $\xi_{k+1} = \xi_k - \gamma_k F_{i_k}'(x_k)^* (F_{i_k}(x_k) - y_{i_k})$; 

\item[\emph{(iii)}] Calculate $x_{k+1} = \arg\min_{x\in X} \left\{\R(x) - \l \xi_{k+1}, x\r\right\}$.
\end{enumerate}
\end{algorithm}

When every $F_i$ is a bounded linear operator, Algorithm \ref{alg:SMD} has been proposed and analyzed in 
\cite{JLZ2023}; under Assumption \ref{ass0} on $\R$ the convergence has been established and, when the 
sought solution satisfies a benchmark source condition, a convergence rate has been derived. 
In order to carry out the analysis of Algorithm \ref{alg:SMD} for nonlinear ill-posed system, beside 
Assumption \ref{ass0} on $\R$, we also assume the following standard conditions on $F_i$ for 
$i =1, \cdots, N$. 

\begin{Assumption}\label{ass2}
\begin{enumerate}[leftmargin = 0.9cm]
\item[\emph{(i)}] $X$ is a Banach space, $Y_i$ is a Hilbert space for each $i$.

\item[\emph{(ii)}] There exists $\rho>0$ such that $B_{2\rho}(x_0)\subset \emph{dom}(F)$ and 
(\ref{sys.min}) has a solution $x^\dag$ such that $D_\R^{\xi_0} (x^\dag, x_0) \le \sigma \rho^2$.

\item[\emph{(iii)}] There exists $0\le \eta < 1$ such that 
\begin{align*}
\|F_i(\tilde x) - F_i(x) - F_i'(x) (\tilde x - x)\| \le \eta \|F_i(\tilde x) - F_i(x)\|
\end{align*}
for all $\tilde x, x \in B_{2\rho}(x_0)$ and $i = 1, \cdots, N$.

\item[\emph{(iv)}] There is $\kappa_0\ge 0$ such that for each $i\in \{1, \cdots, N\}$ and each 
$x \in B_{2\rho}(x_0)$, there exists a bounded linear operator $Q_x^i: Y_i \to Y_i$ such that 
\begin{align*}
F_i'(x) = Q_x^i F_i'(x^\dag) \quad \mbox{and} \quad \|I - Q_x^i\| \le \kappa_0 \|x-x^\dag\|.
\end{align*}
\end{enumerate}
\end{Assumption}

Note that when $N = 1$, Assumption \ref{ass2} reduces to Assumption \ref{ass1} and Assumption \ref{ass1.5}. 
Let 
$$
L := \max_{i= 1, \cdots, N} \sup_{x\in B_{2\rho}(x_0)} \|F_i'(x)\|.
$$
From (iv) of Assumption \ref{ass2} it is easy to see $L <\infty$. From simplicity of exposition, 
throughout this section we set $A_i := F_i'(x^\dag)$ and define $A: X \to Y_1\times \cdots \times Y_N$ by 
\begin{align*}
A x := (A_1 x, \cdots, A_N x), \quad x \in X.
\end{align*}
It is easy to see that the adjoint $A^*$ of $A$ is given by 
\begin{align*}
A^* z = \sum_{i=1}^N A_i^* z_i, \quad \forall z:=(z_1, \cdots, z_N) \in Y_1\times \cdots \times Y_N.
\end{align*}
Assumption \ref{ass2} will enable us to derive a convergence rate of Algorithm \ref{alg:SMD} under the 
source condition (\ref{sc0}) of the sought solution $x^\dag$. i.e. there exists  $\la^\dag = (\la_1^\dag, \cdots, \la_N^\dag) \in Y_1\times \cdots \times Y_N$ such that 
\begin{align}\label{sc1}
\xi^\dag := \xi_0 + A^*\la^\dag = \xi_0 + \sum_{i=1}^N A_i^* \la_i^\dag \in \p \R(x^\dag). 
\end{align}

In the following result we first show that Algorithm \ref{alg:SMD} is well-defined and the Bregman 
distance $D_\R^{\xi_k}(x^\dag, x_k)$ is monotonically decreasing. 

\begin{lemma}\label{Land:lem6}
Let Assumption \ref{ass0} and (i)-(iii) of Assumption \ref{ass2} hold with $L<\infty$. Consider Algorithm 
\ref{alg:SMD}. Let $\bar \gamma := \sup_k \gamma_k$ and assume 
\begin{align}\label{sys.2}
c_6 := 1- \eta - \frac{\bar \gamma L^2}{4 \sigma} >0. 
\end{align}
Then $x_k \in B_{2\rho}(x_0)$ and 
\begin{align*}
\Delta_{k+1} - \Delta_k \le - c_6 \gamma_k \|F_{i_k}(x_k) - y_{i_k}\|^2
\end{align*}
for all integers $k \ge 0$, where $\Delta_k:= D_\R^{\xi_k}(x^\dag, x_k)$. 
\end{lemma}

\begin{proof}
When $x_k \in B_{2\rho}(x_0)$, we may use the similar argument in the proof of Lemma \ref{Land:lem1} 
to obtain 
\begin{align*}
\Delta_{k+1} - \Delta_k 
& \le \frac{1}{4\sigma} \|\xi_{k+1} - \xi_k\|^2 + \l \xi_{k+1} - \xi_k, x_k - x^\dag\r \\
& = \frac{\gamma_k^2}{4\sigma} \|F_{i_k}'(x_k)^*(F_{i_k}(x_k) - y_{i_k})\|^2 
- \gamma_k \left\l F_{i_k}(x_k) - y_{i_k}, F_{i_k}'(x_k) (x_k - x^\dag)\right\r \\
& \le \frac{\gamma_k^2}{4\sigma} \|F_{i_k}'(x_k)^*(F_{i_k}(x_k) - y_{i_k})\|^2 
- (1-\eta) \gamma_k \|F_{i_k}(x_k) - y_{i_k}\|^2. 
\end{align*}
By using $\|F_{i_k}'(x_k)\| \le L$, we further obtain 
\begin{align*}
\Delta_{k+1} - \Delta_k 
& \le - \left(1-\eta - \frac{\gamma_k L^2}{4 \sigma}\right) \gamma_k \|F_{i_k}(x_k) - y_{i_k}\|^2 
\le - c_6 \gamma_k \|F_{i_k}(x_k) - y_{i_k}\|^2.
\end{align*}
Based on this, we may use an induction argument, as done in the proof of Lemma \ref{Land:lem1}, 
to conclude the proof.  
\end{proof}

Note that, once $(x_0, \xi_0) \in X\times X^*$ and $\{\gamma_k\}$ are fixed, the sequence 
$\{(x_k, \xi_k)\}$ in Algorithm \ref{alg:SMD} is completely determined by the sample path 
$\{i_0, i_1, \cdots\}$; changing the sample path can result in a different iterative sequence 
and thus $\{(x_k, \xi_k)\}$ is a random sequence. We need to analyze Algorithm \ref{alg:SMD} 
using tools from stochastic calculus. Let $\F_0 = \emptyset$ and, for each integer $k \ge 0$, 
let $\F_k$ denote the $\sigma$-algebra generated by the random variables $i_l$ for $0\le l < k$. 
Then $\{\F_k: k\ge 0\}$ form a filtration which is natural to Algorithm \ref{alg:SMD}. Let $\EE$ 
denote the expectation associated with this filtration, see \cite{B2020}. There holds the tower property 
$\EE[\EE[\Phi|\F_k]] = \EE[\Phi]$ for any random variable $\Phi$. 

Based on Lemma \ref{Land:lem6}, it is easy to derive the convergence of $\{x_k\}$ to a solution 
of (\ref{sys.min}) in expectation and almost surely by the arguments in \cite{JLZ2023} with minor 
modifications. Our focus here is on deriving convergence rate of Algorithm \ref{alg:SMD} under 
the source condition (\ref{sc0}). To this end, we need to interpret Algorithm \ref{alg:SMD} from 
an alternative perspective. Note that, under (iv) of Assumption \ref{ass2}, the updating formula 
for $\xi_{k+1}$ can be written as 
$$
\xi_{k+1} = \xi_k - \gamma_k A_{i_k}^* (Q_{x_k}^{i_k})^* (F_{i_k}(x_k) - y_{i_k}).
$$
This motivates us to consider the following algorithm.

\begin{algorithm}\label{alg:SMD2}
Take $(x_0, \xi_0) \in X\times X^*$ with $x_0 \in \emph{dom}(F)$ and $\xi_0 \in \p \R(x_0)$. Set 
$\la_0 = (0, \cdots, 0) \in Y_1 \times \cdots \times Y_N$. For $k \ge 0$ do the following:
\begin{enumerate}[leftmargin = 0.9cm]
\item[\emph{(i)}] Pick $i_k \in \{1, \cdots, N\}$ randomly via the uniform distribution;

\item[\emph{(ii)}] Update $\la_{k+1} = (\la_{k+1,1}, \cdots, \la_{k+1,N}) \in Y_1 \times \cdots \times Y_N$ by 
$\la_{k+1, i} = \la_{k, i}$ for $i \ne i_k$ and 
$$
\la_{k+1, i_k} = \la_{k, i_k} - \gamma_k (Q_{x_k}^{i_k})^*(F_{i_k}(x_k) - y_{i_k});
$$

\item[\emph{(iii)}] Calculate $\xi_{k+1} = \xi_0 + A^* \la_{k+1}$ and 
$x_{k+1} = \arg\min_{x\in X} \left\{\R(x) - \l \xi_{k+1}, x\r\right\}$.
\end{enumerate}
\end{algorithm}

Note that Algorithm \ref{alg:SMD2} is not an implementable one, it is only used to produce an extra 
sequence $\{\la_k\}$ in $Y_1\times \cdots\times Y_N$ which is crucial for the forthcoming 
analysis. Note that for $\xi_{k+1}$ defined by algorithm \ref{alg:SMD2} we have 
\begin{align*}
\xi_{k+1} & = \xi_0 + A^* \la_{k+1} = \xi_0 + \sum_{i\ne i_k} A_i^* \la_{k+1, i} + A_{i_k}^* \la_{k+1,i_k}\\
& = \xi_0 +  \sum_{i\ne i_k} A_i^* \la_{k,i} 
+ A_{i_k}^* \left(\la_{k,i_k} - \gamma_k (Q_{x_k}^{i_k})^* (F_{i_k}(x_k) - y_{i_k})\right) \\
& = \xi_k - \gamma_k A_{i_k}^* (Q_{x_k}^{i_k})^* (F_{i_k}(x_k) - y_{i_k}) \\
& = \xi_k - \gamma_k F_{i_k}'(x_k)^* (F_{i_k}(x_k) - y_{i_k}).
\end{align*}
Therefore both Algorithm \ref{alg:SMD} and Algorithm \ref{alg:SMD2} can produce the same 
sequence $\{x_k, \xi_k\}$ if they start from the same initial guess, follow the same sample 
path and use the same step-size. 

\begin{lemma}\label{Land:lem7}
Let Assumption \ref{ass0} and Assumption \ref{ass2} hold. Let $\{\gamma_k\}$ satisfy (\ref{sys.2}). 
Consider Algorithm \ref{alg:SMD2} and assume the source condition (\ref{sc1}) holds. Then
\begin{align*}
\EE[\|\la_{k+1} - \la^\dag\|^2|\F_k] 
& \le  \left(1 + \frac{(3+\eta)^2\kappa_0^2}{8\sigma N} \gamma_k \|F(x_k) - y\|^2 \right)\|\la_k - \la^\dag\|^2 \\
& \quad \, + \frac{C \gamma_k^2}{N} \left\|F(x_k) - y\right\|^2 - \frac{2\gamma_k}{N} \Delta_k
\end{align*}
for all integers $k \ge 0$, where $\Delta_k := D_\R^{\xi_k}(x^\dag, x_k)$ and $C$ is a positive constant 
independent of $k$. 
\end{lemma}

\begin{proof}
By the definition of $\la_{k+1}$ and the polarization identity we have 
\begin{align*}
\|\la_{k+1} - \la^\dag\|^2 
& = \sum_{i\ne i_k} \|\la_{k+1, i} - \la_i^\dag\|^2 + \|\la_{k+1, i_k} - \la_{i_k}^\dag\|^2 \\
& = \sum_{i\ne i_k} \|\la_{k, i} - \la_i^\dag\|^2 
+ \| \la_{k, i_k} - \la_{i_k}^\dag - \gamma_k (Q_{x_k}^{i_k})^* (F_{i_k}(x_k) - y_{i_k})\|^2 \displaybreak[0]\\
& = \|\la_k - \la^\dag\|^2 + \gamma_k^2 \left\|(Q_{x_k}^{i_k})^* (F_{i_k}(x_k) - y_{i_k})\right\|^2 \\
& \quad \, - 2 \gamma_k \left\l (Q_{x_k}^{i_k})^* (F_{i_k}(x_k) - y_{i_k}), \la_{k,i_k} - \la_{i_k}^\dag\right\r.
\end{align*}
By using Assumption \ref{ass2} and following the similar argument in establishing Lemma \ref{Land:lem2} 
we can obtain 
\begin{align*}
\|\la_{k+1} - \la^\dag\|^2 
& \le  \|\la_k - \la^\dag\|^2 + \gamma_k^2 (1 + \kappa_0 \|x_k - x^\dag\|)^2 \left\|F_{i_k}(x_k) - y_{i_k}\right\|^2 \\
& \quad \, + 2 \kappa_0 \gamma_k \|x_k - x^\dag\| \|F_{i_k}(x_k) - y_{i_k}\| \|\la_{k,i_k} - \la_{i_k}^\dag\| \\
& \quad \, - 2 \gamma_k \left\l F_{i_k}(x_k) - y_{i_k} - A_{i_k} (x_k - x^\dag), \la_{k,i_k} - \la_{i_k}^\dag\right\r \\
& \quad \, - 2 \gamma_k \left\l A_{i_k} (x_k - x^\dag), \la_{k, i_k} - \la_{i_k}^\dag\right\r \\
& \le  \|\la_k - \la^\dag\|^2 + \gamma_k^2 (1 + \kappa_0 \|x_k - x^\dag\|)^2 \left\|F_{i_k}(x_k) - y_{i_k}\right\|^2 \\
& \quad \, + (3+\eta) \kappa_0 \gamma_k \|x_k - x^\dag\| \|F_{i_k}(x_k) - y_{i_k}\| \|\la_{k,i_k} - \la_{i_k}^\dag\| \\
& \quad \, - 2 \gamma_k \left\l A_{i_k}^*(\la_{k, i_k} - \la_{i_k}^\dag), x_k - x^\dag\right\r.
\end{align*}
According to Lemma \ref{Land:lem6} and the strong convexity of $\R$, $\{x_k\}$ is bounded. Thus, we can 
find a positive constant $C$ such that 
\begin{align*}
\|\la_{k+1} - \la^\dag\|^2 
& \le  \|\la_k - \la^\dag\|^2 + C \gamma_k^2 \left\|F_{i_k}(x_k) - y_{i_k}\right\|^2 \\
& \quad \, + (3+\eta) \kappa_0 \gamma_k \|x_k - x^\dag\| \|F_{i_k}(x_k) - y_{i_k}\| \|\la_{k,i_k} - \la_{i_k}^\dag\| \\
& \quad \, - 2 \gamma_k \left\l A_{i_k}^*(\la_{k, i_k} - \la_{i_k}^\dag), x_k - x^\dag\right\r.
\end{align*}
Taking the conditional expectation on $\F_k$ gives 
\begin{align*}
\EE[\|\la_{k+1} - \la^\dag\|^2|\F_k] 
& \le  \|\la_k - \la^\dag\|^2 + \frac{C \gamma_k^2}{N} \sum_{i=1}^N \left\|F_i(x_k) - y_i\right\|^2 \\
& \quad \, + \frac{(3+\eta)\kappa_0}{N} \gamma_k \|x_k - x^\dag\| \sum_{i=1}^N \|F_i(x_k) - y_i\| \|\la_{k,i} - \la_i^\dag\| \\
& \quad \, - \frac{2 \gamma_k}{N} \left\l \sum_{i=1}^N A_i^*(\la_{k, i} - \la_i^\dag), x_k - x^\dag\right\r.
\end{align*}
By the Cauchy-Schwarz inequality, we have 
\begin{align*} 
\sum_{i=1}^N \|F_i(x_k) - y_i\| \|\la_{k,i} - \la_i^\dag\| 
& \le \left(\sum_{i=1}^N \|F_i(x_k)-y_i\|^2\right)^{1/2} \left(\sum_{i=1}^N \|\la_{k,i}- \la_i^\dag\|^2\right)^{1/2}\\
& = \|F(x_k) - y\| \|\la_k - \la^\dag\|. 
\end{align*}
Combining this with the above inequality and using the definition of $F$ and $A$, we thus obtain 
\begin{align*}
\EE[\|\la_{k+1} - \la^\dag\|^2|\F_k] 
& \le  \|\la_k - \la^\dag\|^2 + \frac{C \gamma_k^2}{N}  \left\|F(x_k) - y\right\|^2 \\
& \quad \, + \frac{(3+\eta)\kappa_0}{N} \gamma_k \|x_k - x^\dag\| \|F(x_k) - y\| \|\la_k - \la^\dag\| \\
& \quad \, - \frac{2 \gamma_k}{N} \left\l A^*(\la_k - \la^\dag), x_k - x^\dag\right\r.
\end{align*}
By using the definition of $\xi_k$, the source condition (\ref{sc1}), and the strong convexity 
of $\R$, we have 
$$
\left\l A^*(\la_k - \la^\dag), x_k - x^\dag\right\r
= \l \xi_k - \xi^\dag, x_k - x^\dag\r 
\ge \Delta_k + \sigma \|x_k - x^\dag\|^2. 
$$
Therefore 
\begin{align*}
\EE[\|\la_{k+1} - \la^\dag\|^2|\F_k] 
& \le  \|\la_k - \la^\dag\|^2 + \frac{C \gamma_k^2}{N}  \left\|F(x_k) - y\right\|^2 
- \frac{2 \sigma \gamma_k}{N} \|x_k - x^\dag\|^2 - \frac{2\gamma_k}{N} \Delta_k \\
& \quad \, + \frac{(3+\eta)\kappa_0}{N} \gamma_k \|x_k - x^\dag\| \|F(x_k) - y\| \|\la_k - \la^\dag\| \\
& \le  \|\la_k - \la^\dag\|^2 + \frac{C \gamma_k^2}{N}  \left\|F(x_k) - y\right\|^2 - \frac{2\gamma_k}{N} \Delta_k\\
& \quad \, + \frac{(3+\eta)^2 \kappa_0^2}{8\sigma N} \gamma_k \|F(x_k) - y\|^2 \|\la_k - \la^\dag\|^2. 
\end{align*}
The proof is complete. 
\end{proof}

We will use Lemma \ref{Land:lem6} and Lemma \ref{Land:lem7} to derive the almost sure convergence rate 
of Algorithm \ref{alg:SMD}. To this end, we need the following Robbins-Siegmund theorem (\cite{B2016,RS1971}).

\begin{theorem}[Robbins-Siegmund theorem] \label{theorem_RS}
In a probability space consider a filtration $\{\F_k\}$ and four non-negative 
sequences of $\{\F_k\}$-adapted processes $\{V_k\}$, $\{U_k\}$, $\{Z_k\}$ and $\{\a_k\}$ such that 
$\sum_k Z_k < \infty$ and $\sum_k \alpha_k < \infty$ almost surely. If for any integers 
$k\ge 0$, there holds 
\begin{equation*}
\EE[V_{k+1} \mid \F_k] + U_k \leq (1+\a_k) V_k + Z_k,
\end{equation*}
then $\{V_k\}$ converges and $\sum_{k=0}^{\infty} U_k$ is finite almost surely.
\end{theorem}

Now we are ready to show the main convergence rate result on Algorithm \ref{alg:SMD} under the 
source condition (\ref{sc1}). 

\begin{theorem}\label{Land:thm4}
Let Assumption \ref{ass0} and Assumption \ref{ass2} hold. Let $\{\gamma_k\}$ be a sequence of 
positive numbers such that 
\begin{align}\label{gamma4}
\bar \gamma : = \sup_{k\ge 0} \gamma_k < \frac{4\sigma (1-\eta)}{L^2} \quad \emph{ and } \quad 
\sum_{k=0}^\infty \gamma_k = \infty. 
\end{align}
Consider Algorithm \ref{alg:SMD} and assume the source condition (\ref{sc1}) holds. Then 
\begin{align*}
\Delta_k = O\left(s_k^{-1} \right) \quad \mbox{ and } \quad 
\|x_k - x^\dag\| = O\left(s_k^{-1/2}\right) \quad \emph{almost surely},
\end{align*}
where $s_k := \sum_{l=0}^k\gamma_l$. 
\end{theorem}

\begin{proof}
According to Lemma \ref{Land:lem6}, taking the conditional expectation on $\F_k$ gives
$$
\EE[\Delta_{k+1}|\F_k] \le \Delta_k - \frac{c_6 \gamma_k}{N} \|F(x_k) - y\|^2.  
$$
By taking the full expectation, we further obtain 
\begin{align*}
\frac{c_6 \gamma_k}{N} \EE[\|F(x_k) - y\|^2] \le \EE[\Delta_k] - \EE[\Delta_{k+1}].
\end{align*}
Consequently 
\begin{align*}
\frac{c_6}{N} \EE\left[\sum_{k=0}^\infty \gamma_k \|F(x_k) - y\|^2\right] 
\le \sum_{k=0}^\infty \left(\EE[\Delta_k] - \EE[\Delta_{k+1}]\right) 
\le \Delta_0 < \infty.
\end{align*}
This implies that $\sum_{k=0}^\infty \gamma_k \|F(x_k) - y\|^2 < \infty$ almost surely. 
Therefore, we may apply Robbins-Siegmund theorem to the inequality in Lemma \ref{Land:lem7}
to conclude that $\sum_{k=0}^\infty \gamma_k \Delta_k < \infty$ almost surely. By using 
the monotonicity of $\{\Delta_k\}$ given in Lemma \ref{Land:lem6} we obtain 
\begin{align*}
s_k \Delta_k = \left(\sum_{l=0}^k\gamma_l\right) \Delta_k 
& \le \sum_{l=0}^k \gamma_l \Delta_l \le \sum_{l=0}^\infty \gamma_l \Delta_l < \infty 
\quad \mbox{almost surely}.
\end{align*}
Therefore 
\begin{align*}
\Delta_k = O\left(s_k^{-1}\right) \quad \mbox{almost surely}.
\end{align*}
By the strong convexity of $\R$, we then obtain $\|x_k - x^\dag\| = O(s_k^{-1/2})$ almost 
surely. The proof is complete. 
\end{proof}

\begin{remark}
{\rm
When $X$ is a Hilbert space and $\R(x) = \frac{1}{2} \|x\|^2$, Algorithm \ref{alg:SMD} reduces 
to the stochastic gradient method studied in \cite{JZZ2020}. Under the diminishing step-size 
$\gamma_k = \gamma_0(k+1)^{-\a}$ with $\a\in (0,1)$, \cite[Theorem 4.8]{JZZ2020} shows that, 
under the source condition (\ref{sc1}), there holds the convergence rate
$$
\EE[\|x_k - x^\dag\|^2] = O\left((k+1)^{-\min\{1-\a, \a-\varepsilon\}}\right) 
$$
for some small $\varepsilon\in (0, \a/2)$. Besides Assumption \ref{ass2} and smallness of $\|\la^\dag\|$
and $\gamma_0$, the proof in \cite{JZZ2020} also relies on the technical condition
\begin{align*}
\EE\left[\|(I - R_{z_t}) G(x_k)\|^2\right]^{1/2} 
&\le C \EE\left[\|x_k-x^\dag\|^2\right]^{\theta/2}\EE\left[\|G(x_k)\|^2\right]^{1/2}, \\
\EE\left[\|(I - R_{z_t}^*) G(x_k)\|^2\right]^{1/2} 
&\le C \EE\left[\|x_k-x^\dag\|^2\right]^{\theta/2}\EE\left[\|G(x_k)\|^2\right]^{1/2}
\end{align*}
for some $\theta\in (0, 1]$; one may refer to \cite[Assumption 2.4]{JZZ2020} for the meaning of 
$R_{z_t}$ and $G(x)$. 
Unfortunately there is no any justification available for this technical condition. In contrast,
our result requires neither such a technical condition nor the smallness of $\|\la^\dag\|$. 
Our convergence rate is valid for any step-size $\{\gamma_k\}$ satisfying (\ref{gamma4}). 
For the step-size scheme $\gamma_k = \gamma_0 (k+1)^{-\a}$ with $\a \in (0,1)$, our Theorem 
\ref{Land:thm4} gives the nicer almost sure convergence rate 
$$
\|x_k- x^\dag\|^2 = O\left((k+1)^{-(1-\a)}\right).
$$
Our result even allows to use constant step-size schemes to give much better almost sure convergence rate
$$
\|x_k- x^\dag\|^2 = O\left(1/(k+1)\right).
$$
}
\end{remark}

\begin{remark}
{\rm For the stochastic mirror descent method for nonlinear ill-posed system with noisy data, it is not yet clear how to derive convergence rates. We hope to address this issue in a future work. }
\end{remark}

\section{\bf Numerical results}\label{sect5}

In this section we present some computational results to validate our theoretical convergence rate 
results on the method (\ref{Land}) when it is terminated by the discrepancy principle (\ref{DP0}) with 
$\tau > (1+\eta)/(1-\eta)$. The step-sizes $\gamma_k^\d$ are chosen by various rules which are described  
below:

\begin{enumerate}[leftmargin = 0.8cm]
\item[$\bullet$] {\it Rule 1}: $\gamma_k^\d = \gamma/L^2$ with $\gamma = 1.98(1-\eta - (1+\eta)/\tau)$;

\item[$\bullet$] {\it Rule 2}: $\gamma_k^\d 
= \max\left\{\frac{\gamma \|r_k^\d\|^2}{\|F'(x_k^\d)^*r_k^\d\|^2}, \bar \gamma\right\}$ 
with $\gamma = 1.98(1-\eta - (1+\eta)/\tau)$ and $\bar \gamma = 600$;

\item[$\bullet$] {\it Rule 3}: $\gamma_k^\d 
= \max\left\{\frac{\gamma_0((1-\eta) \|r_k^\d\| - (1+\eta) \d)\|r_k^\d\|}{\|F'(x_k^\d)^*r_k^\d\|^2}, \bar \gamma\right\}$
with $\gamma_0 = 1.98$ and $\bar \gamma = 600$. 
\end{enumerate}
Here $r_k := F(x_k^\d) - y^\d$. In the following computations, the regularization functionals $\R$ that
we use always satisfy Assumption \ref{ass0} with $\sigma = 1/2$. Thus, the step-size selections for 
$\gamma_k^\d$, as outlined above, are legitimate in accordance with Corollary \ref{Land:cor1} and 
Theorem \ref{Land:thm3}.  

\begin{example}\label{ex:entropy}
{\rm
Consider the linear integral equation of the first kind
\begin{align*}
(Ax)(s) := \int_0^1 \phi(t,s) x(t) dt = y(s), \quad s\in [0,1],
\end{align*}
where $\phi(t,s) := 1+t+s$, and assume the sought solution $x^\dag$ is a probability density function
on $[0,1]$, i.e. $x^\dag \ge 0$ a.e. on $[0,1]$ and $\int_0^1 x^\dag = 1$. Clearly $A$ is a compact linear 
operator from $L^1[0,1]$ to $L^2[0,1]$ and 
\begin{align*}
\|Ax\|_{L^2[0,1]}^2 \le \int_0^1 \max_{0\le t\le 1} |\phi(t,s)|^2 dt \|x\|_{L^1[0,1]}^2 
= \frac{19}{3} \|x\|_{L^1[0,1]}^2
\end{align*} 
which implies $\|A\| \le L:=\sqrt{19/3}$. To determine such a solution we use the regularization functional 
$\R$ given in (d) of Example \ref{ex1} with $\Omega = [0,1]$. The corresponding method (\ref{Land}) becomes 
(\cite{Jin2022})
\begin{align}\label{Land:entropy}
\xi_{k+1}^\d = \xi_k^\d - \gamma_k^\d A^*(A x_k^\d - y^\d), \qquad 
x_{k+1}^\d = e^{\xi_{k+1}^\d}\left/\int_0^1 e^{\xi_{k+1}^\d},\right.
\end{align}
where $y^\d$ is a noisy data satisfying $\|y^\d - A x^\dag\|_{L^2[0,1]} \le \d$ with a noise level $\d>0$.
Clearly, Assumption \ref{ass1} and Assumption \ref{ass1.5} hold with $\eta = 0$ and $\kappa_0 = 0$. 
It is easy to check that if $x^\dag$ satisfies the condition 
\begin{align}\label{sc:entropy}
1 + \log x^\dag \in \mbox{Ran}(A^*),
\end{align}
then it satisfies the source condition (\ref{sc0}) with $\xi_0 = 0$. Consequently, for the integer 
$k_\d$ output by the discrepancy principle, if $x^\dag$ satisfies (\ref{sc:entropy}) then we can 
expect the convergence rate $\|x_{k_\d}^\d - x^\dag\| = O(\sqrt{\d})$ according to our theoretical 
results. 

\begin{table}[ht]
\caption{Numerical results for Example \ref{ex:entropy}, where \texttt{iter} denotes the integer $k_\d$ 
output by the discrepancy principle (\ref{DP0}) with $\tau = 1.01$, \texttt{err} denotes the approximation 
error, i.e. $\texttt{err} = \|x_{k_\d}^\d - x^\dag\|_{L^1}$, and \texttt{ratio} = \texttt{err}/$\sqrt{\d}$.} \label{table1}
\begin {center}
\begin{tabular}{|c|ccc|ccc|ccc|}
    \hline
&  \multicolumn{3}{c|}{Rule 1} 
& \multicolumn{3}{c|}{Rule 2} & \multicolumn{3}{c|}{Rule 3} \\ \cline{2-10}
$\d$  & \texttt{iter} & \texttt{err} & \texttt{ratio}
& \texttt{iter} & \texttt{err} & \texttt{ratio}
& \texttt{iter} & \texttt{err} & \texttt{ratio}\\
\hline
5e-2   & 743  & 5.0723e-2  & 0.2268  & 160  & 5.0659e-2 & 0.2266 & 66  & 5.0275e-2 & 0.2248\\
5e-3   & 2630 & 5.0405e-3  & 0.0713  & 1649 & 5.0215e-3 & 0.0710 & 189 & 4.9796e-3 & 0.0704\\
5e-4   & 4509 & 4.6703e-4  & 0.0209  & 3495 & 4.6610e-4 & 0.0208 & 280 & 4.6441e-4 & 0.0208\\
5e-5   & 6387 & 8.2451e-5  & 0.0117  & 5348 & 8.2431e-5 & 0.0117 & 368 & 8.2427e-5 & 0.0117\\
\hline
\end{tabular}
\end{center}
\end{table}

To numerically check if the convergence rate is achieved, we assume the sought solution $x^\dag$ is 
given by 
$$
x^\dag(t) := e^{1.5a-1+at}, \quad t \in [0.1],
$$
where $a = 0.4949075935\cdots$ is such that $e^{1.5a-1} (e^a-1)/a = 1$ which guarantees $\int_0^1 x^\dag = 1$
so that $x^\dag$ is a probability density function on $[0,1]$. It is easy to see that 
$1 + \log x^\dag = A^* \la^\dag$ with $\la^\dag \equiv a$ on $[0,1]$ and thus the source condition 
(\ref{sc:entropy}) holds. In our numerical computation, we use noisy data $y^\d$ with various noise
level $\d>0$ to reconstruct $x^\dag$. To carry out the computation, all integrals over $[0,1]$ are 
approximated by the trapezoidal rule by partitioning $[0,1]$ into $5000$ subintervals of equal length;
such a fine grid is used for the purpose of reducing the effect of discretization so that we can 
observe the convergence rate in terms of $\d$ more accurately. In Table \ref{table1} we present the 
computational results obtained using the method (\ref{Land:PDE}) with the initial guess $\xi_0 \equiv 0$,
leading to $x_0 \equiv 1$. The method is terminated according to the discrepancy principle (\ref{DP0}) 
with $\tau = 1.01$, and the step-size $\gamma_k^\d$ is selected using Rule 1, Rule 2 and Rule 3. The 
computational outcomes corroborate our theoretical findings discussed in Corollary \ref{Land:cor1} and 
Theorem \ref{Land:thm3}. Moreover, Table \ref{table1} indicates that Rule 3 generally requires fewer 
iterations than Rule 1 and Rule 2. 
}
\end{example}

\begin{example}\label{ex:PDE}
{\rm
Let $\Omega\subset {\mathbb R}^2$ be a bounded domain with a Lipschitz boundary $\p \Omega$. 
Consider determining the coefficient $c$ in the boundary value problem
\begin{align}\label{PDE}
- \triangle u + c u = f \,\,\, \mbox{in } \Omega, \quad u = g \mbox{ on } \p \Omega
\end{align}
from an $L^2(\Omega)$-measurement of $u$, where $f\in H^{-1}(\Omega)$ and $g \in H^{1/2}(\p \Omega)$.
Assuming the sought solution $c^\dag$ is in $L^2(\Omega)$, this nonlinear inverse problem reduces to 
solving $F(c) = u$, where $F: L^2(\Omega) \to L^2(\Omega)$ is defined by $F(c) := u(c)$ with 
$u(c) \in H^1(\Omega) \subset L^2(\Omega)$ being the unique weak solution of (\ref{PDE}). It is known that 
$F$ is well-defined on the set 
$$
D(F) := \{c \in L^2(\Omega): \|c - \hat c\|_{L^2} \le \varepsilon_0 \mbox{ for some } \hat c \ge 0 \mbox{ a.e.}\}
$$
for some positive constant $\varepsilon_0>0$. Furthermore, $F$ is Fr\'{e}chet differentiable, the Fr\'{e}chet 
derivative of $F$ and the adjoint are given by 
$$
F'(c) h = - A(c)^{-1} (h u(c))    \quad \mbox{ and } \quad F'(c)^* w = - u(c) A(c)^{-1} w
$$
for $c\in D(F)$ and $h, w \in L^2(\Omega)$, where $A(c): H_0^1(\Omega) \to H^{-1}(\Omega)$ is defined by 
$A(c) u = - \triangle u  + c u$. Additionally, $F$ satisfies Assumption \ref{ass1} (iii) for a number 
$\eta>0$ which can be very small if $\rho>0$ is sufficiently small (see \cite{EHN1996}). If 
$u(c^\dag) \ge \kappa$ for some constant $\kappa>0$, then Assumption \ref{ass1.5} holds in a neighborhood 
of $c^\dag$; see \cite{HNS1995}.  
 
We are interested in the situation that the sought solution $c^\dag \in L^2(\Omega)$ is nonnegative.
Thus, to reconstruct $c^\dag$, we use the functional $\R$ given in (a) of Example \ref{ex1} with 
$C := \{c\in L^2(\Omega): c\ge 0 \mbox{ a.e.}\}$. Correspondingly the method (\ref{Land}) becomes 
\begin{align}\label{Land:PDE}
\xi_{k+1}^\d = \xi_k^\d - \gamma_k^\d F'(c_k^\d)^*(F(c_k^\d) - u^\d), \qquad 
c_{k+1}^\d = \max\{\xi_{k+1}^\d, 0\},
\end{align}
where $u^\d\in L^2(\Omega)$ is a noisy data satisfying $\|u^\d - u(c^\dag)\|_{L^2(\Omega)} \le \d$
with a noise level $\d>0$. The source condition (\ref{sc0}) can be equivalently stated as 
\begin{align}\label{sc:PDE}
c^\dag = P_C(\xi_0 + F'(c^\dag)^* \la^\dag) = \max\{\xi_0 + F'(c_0)^* \la^\dag, 0\}
\end{align}
for some $\la^\dag \in L^2(\Omega)$, where $P_C$ denotes the orthogonal projection from $L^2(\Omega)$ 
onto $C$. 

\begin{table}[ht]
\caption{Numerical results for Example \ref{ex:PDE}, where \texttt{iter} denotes the integer $k_\d$ output 
by the discrepancy principle (\ref{DP0}) with $\tau = 1.1$, \texttt{err} denotes the approximation error, 
i.e. $\texttt{err} = \|c_{k_\d}^\d - c^\dag\|_{L^2}$, and \texttt{ratio} = \texttt{err}/$\sqrt{\d}$.} \label{table2}
\begin {center}
\begin{tabular}{|c|ccc|ccc|}
    \hline
&  \multicolumn{3}{c|}{Rule 2} 
& \multicolumn{3}{c|}{Rule 3}  \\ \cline{2-7}
$\d$  & \texttt{iter} & \texttt{err} & \texttt{ratio}
& \texttt{iter} & \texttt{err} & \texttt{ratio}  \\
\hline
1e-2   & 19   & 1.7471e-1  & 1.7471  & 2    & 1.7357e-1 & 1.7357 \\
1e-3   & 97   & 4.0901e-2  & 1.2934  & 26   & 4.1922e-2 & 1.3257 \\
1e-4   & 227  & 8.3865e-3  & 0.8386  & 84   & 8.3326e-3 & 0.8333 \\
1e-5   & 548  & 2.5209e-3  & 0.7972  & 365  & 2.5909e-3 & 0.8193 \\
1e-6   & 5423 & 9.3739e-4  & 0.9374  & 5741 & 9.4716e-3 & 0.9472 \\
\hline
\end{tabular}
\end{center}
\end{table}

To test the convergence rate result, we take $\Omega = [0,1]^2$ and assume the sought solution $c^\dag$ 
is given by  
$$
c^\dag (x,y) = \left(\max\{1-9(x^2+y^2), 0\}\right)^2
$$
and the right hand side $f$ is given by $f = -4 + (1+x^2+y^2) c^\dag(x,y)$. Then $u(c^\dag) = 1 + x^2 + y^2$. 
Take the initial guess $\xi_0 = 0$. Then the source condition (\ref{sc:PDE}) holds with 
\begin{align*}
\la^\dag = A(c^\dag) \left(-\frac{c^\dag}{u(c^\dag)}\right) 
= \triangle \left(\frac{c^\dag}{u(c^\dag)}\right) - \frac{(c^\dag)^2}{u(c^\dag)} \in L^2(\Omega).
\end{align*}
Thus, according to our theory, if the step-size $\gamma_k^\d$ is chosen according to those rules stated 
in Corollary \ref{Land:cor1} and Theorem \ref{Land:thm3}, and if $k_\d$ denotes the integer output 
by the discrepancy principle, we can expect $\|c_{k_\d}^\d - c^\dag\|_{L^2(\Omega)} = O(\sqrt{\d})$
as $\d \to 0$. To validate this result, we use noisy data $u^\d$ with various noise level $\d>0$ to 
reconstruct $c^\dag$ by the method (\ref{Land:PDE}). To perform the computation, we set $\eta = 0.04$ and 
divide $\Omega$ into $128\times 128$ small squares of equal size. All partial differential equations are 
solved approximately using a multigrid method (\cite{H2016}) with finite difference discretization. In 
Table \ref{table2} we report the computational results by the method (\ref{Land:PDE}), terminated by 
the discrepancy principle (\ref{DP0}) with $\tau = 1.1$, with the step-size $\gamma_k^\d$ chosen by 
Rule 2 and Rule 3; we do not report the computation result for $\gamma_k^\d$ chosen by Rule 1 because 
the upper bound $L$ is hard to obtain. The computational results confirm our theoretical result presented 
in Corollary \ref{Land:cor1} and Theorem \ref{Land:thm3}. Furthermore, Table \ref{table2} indicates 
that Rule 3 usually requires less number of iterations than Rule 2 particularly when the noise level 
is moderate or not too small.
}
\end{example}

\section*{\bf Acknowledgement} 

The work of Q. Jin is partially supported by the Future Fellowship of the Australian Research Council (FT170100231).

\bibliographystyle{amsplain}

\end{document}